\documentclass[10pt]{amsart}
\usepackage{amsmath,amssymb,amsthm}
\usepackage{enumerate}             
\usepackage{cite}                  
\usepackage{xspace}                
\usepackage[margin=1in]{geometry}  
\usepackage{mathtools}             
\usepackage{tensor}                
\usepackage[numbered]{bookmark}    

\newtheorem{theorem}{Theorem}[section]
\newtheorem*{theorem-nn}{Theorem}
\newtheorem{corollary}[theorem]{Corollary}
\newtheorem{lemma}[theorem]{Lemma}
\newtheorem{proposition}[theorem]{Proposition}

\newtheorem*{question-nn}{Question}
\theoremstyle{definition}
\newtheorem{definition}[theorem]{Definition}
\theoremstyle{remark}
\newtheorem{remark}[theorem]{Remark}
\newtheorem{example}[theorem]{Example}

\makeatletter
\def\norm{\@ifnextchar[{\norm@i}{\norm@i[]}}
\def\norm@i[#1]{\@ifnextchar[{\norm@ii[#1]}{\norm@ii[#1][]}}
\def\norm@ii[#1][#2]#3{%
\ifx\\#3\\
  \mathrm{N}^{#1}_{#2}
\else
  \mathrm{N}^{#1}_{#2}(#3)
  \fi}

\newsavebox\boxFormula
\newsavebox\boxOverline
\newlength\lenFormula
\newlength\lenOverline

\def\closure{\@ifnextchar[{\closure@i}{\closure@i[0]}}
\def\closure@i[#1]{\@ifnextchar[{\closure@ii[#1]}{\closure@ii[#1][0]}}
\def\closure@ii[#1][#2]#3{
\sbox{\boxFormula}{$\m@th#3$}%
\setbox\boxOverline\null%
\ht\boxOverline=\ht\boxFormula%
\dp\boxOverline=\dp\boxFormula%
\setlength\lenOverline{\the\wd\boxFormula}%
\advance\lenOverline by -#1pt%
\advance\lenOverline by -#2pt%
\wd\boxOverline=\lenOverline%
\sbox\boxOverline{$\m@th\overline{\copy\boxOverline}$}%
\setlength\lenFormula{\the\wd\boxFormula}%
\addtolength\lenFormula{-\the\wd\boxOverline}%
\rlap{\hskip #1pt \usebox\boxOverline}{\usebox\boxFormula}%
}
\makeatother

\newcommand{\invx}{\closure[2][0.3]{X}}
\newcommand{\invy}{\closure[0][0]{Y}}
\newcommand{\invz}{\closure[2][0]{Z}}
\newcommand{\invxcupy}{\closure[2][0]{X \cup Y}}
\newcommand{\clF}{\closure[1.7][0]{F}}
\newcommand{\clG}{\closure[1.5][0]{G}}

\newcommand{\Ast}{\mathop{\scalebox{1.8}{\raisebox{-0.2ex}{$\ast$}}}}%
\newcommand{\cansc}[1]{\mathcal{S}_{#1}}
\newcommand{\diam}{\mathrm{diam}}
\newcommand{\gilcube}{[0,1]^{\omega}}
\newcommand{\hker}[1]{\mathrm{K}_{#1}}
\newcommand{\homgilcube}{\mathrm{Homeo}\bigl( \gilcube \bigr)}
\newcommand{\frpr}[2]{\tensor*[_{#1}]{*}{_{#2}}}
\newcommand{\frprod}[2]{G\! \frpr{\phi_{#1}}{\phi_{#2}}\! H}
\newcommand{\roots}[1]{#1^{\vn}}
\newcommand{\vn}{\varnothing}
\newcommand{\word}[1]{\mathrm{W}(#1)}

\begin{document}
\title{Graev ultrametrics and free products of Polish groups}

\author{Konstantin Slutsky}
\address{Institut for Matematiske Fag \\
  K\o benhavns Universitet\\
  Universitetspaken 5\\
  2100 K\o benhavn \O \\
  Denmark}
\email{kslutsky@gmail.com}

\thanks{Research supported by Denmark's Council for Independent Research (Natural Sciences Division), grant
  no. 10-082689/FNU}
\begin{abstract}
  We construct Graev ultrametrics on free products of groups with two-sided invariant ultrametrics and HNN
  extensions of such groups.  We also introduce a notion of a free product of general Polish groups and prove, in
  particular, that two Polish groups \( G \) and \( H \) can be embedded into a Polish group \( T \) in such a way that
  the subgroup of \( T \) generated by \( G \) and \( H \) is isomorphic to the free product \( G*H \).
\end{abstract}

\maketitle

\section{Introduction}
\label{sec:introduction}

Mark Graev~\cite{MR0038357} gave a construction of two-sided invariant metrics on free groups which now bear his name.
Starting from a pointed metric space \( (X,d,e) \), the Graev metric \( \delta \) is a two-sided invariant metric on the
free group \( F(X) \).  The precise construction of \( \delta \) will be explained below, but it is characterized by
being \emph{the largest two-sided invariant metric on \( F(X) \) that extends \( d \),} where we view \( X \) as being
embedded into \( F(X) \) in a natural way.  In the group theory free groups are important, among other reasons, as
surjectively universal objects: any group is a factor of a free group.  Graev metrics and their generalizations proved
to be very useful in constructing surjectively universal objects in various classes of metrizable groups.

For instance, let \( \mathbb{N}^{\mathbb{N}} \) denote the \emph{Baire space:} the space of infinite sequences of
natural numbers with the metric
\[ d(x,y) = \sup \{\, 2^{-n} \mid n \in \mathbb{N},\ x(n) \ne y(n) \,\}. \]
Let \( \clF(\mathbb{N}^{\mathbb{N}}) \) be the group completion of \( F(\mathbb{N}^{\mathbb{N}}) \)
endowed with the Graev metric (with respect to any distinguished point).

\begin{theorem-nn}[Folklore, see Theorem 2.11 in \cite{MR1288299}]
  \label{thm:Baire-Graev-tsi-universal}
  The group \( \clF(\mathbb{N}^{\mathbb{N}}) \) is surjectively universal in the class of Polish groups that
  admit compatible two-sided invariant metrics.
\end{theorem-nn}
An important question raised in \cite{MR1288299} and further advertised in \cite{MR1425877} is whether there is a
universal Polish group.  Motivated by this question, L.~Ding and S.~Gao \cite{MR2278689} constructed generalized Graev
metrics, and based on this construction Ding \cite{MR2970459} answered the question of Kechris in the affirmative.  In a
recent paper Gao \cite{MR3028617} addressed the question of the existence of surjectively universal Polish
\emph{ultrametric\/} groups and gave yet another modification of Graev's original definition.  The latter paper of Gao
motivates our study of the Graev ultrametrics on free products of ultrametric groups.

\subsection{Main results}
\label{sec:main-results}

The main results of this work are twofold.  In Section \ref{sec:graev-ultr-free} we give the constructions of Graev
ultrametrics for free products and HNN extensions of groups with two-sided invariant ultrametrics.  In particular we
prove
\begin{theorem-nn}[see Theorem \ref{thm:Graev-ultrametric}]
  Let \( (G,d_{G}) \) and \( (H,d_{H}) \) be groups with two-sided invariant ultrametrics, and let \( A = G \cap H \) be
  a common closed subgroup.  There exists a two-sided invariant ultrametric on the free product with
  amalgamation \( G*_{A}H \) that extends ultrametrics \( d_{G} \) and \( d_{H} \).
\end{theorem-nn}

\begin{theorem-nn}[see Theorem \ref{thm:graev-metrics-on-hnn-extensions}]
  Let \( (G, d\,) \) be a group with a two-sided invariant ultrametric \( d \), \( A \) and \( B \) be closed
  subgroups of \( G \) and \( \phi : A \to B \) be an isometric isomorphism.  If \( \diam(A) \le K \), then there
  exists a two-sided invariant ultrametric \( \delta \) on the HNN extension \( H \) of \( (G, \phi) \) which extends \(
  d \) and such that \( \delta(t, e) = K \), where \( t \) is the stable letter of \( H \).
\end{theorem-nn}

While we follow closely the methods of \cite{1111.1538}, the formalism for trivial words used in this paper is
different.  We introduce a new notion of a maximal evaluation forest and argue that it provides a more unified tool for
studying Graev metrics on free products than the notion of an evaluation tree.

In Section \ref{sec:free-products-polish} we step outside of the two-sided invariant world and define a notion of a free
product of general Polish groups.  The results of Section \ref{sec:free-products-polish} are new for both the metric and
the ultrametric settings.  Among other things we prove that

\begin{theorem-nn}[see Theorem \ref{thm:Polish-groups-generate-free-product}]
  Let \( G \) and \( H \) be Polish groups.  There are a Polish group \( T \) and embeddings
  \( \psi_{G} : G \hookrightarrow T \), \( \psi_{H} : H \hookrightarrow T \) such that
  \( \langle \psi_{G}(G), \psi_{H}(H) \rangle \) is naturally isomorphic to the free product \( G * H \).  Moreover, if
  \( G \) and \( H \) admit compatible left invariant ultrametrics, then \( T \) can be chosen to admit such a metric as
  well.
\end{theorem-nn}

\subsection{Notions and notations}
\label{sec:notions-notations}

To establish the terminology, recall that an \emph{ultrametric space\/} is a metric space \( (X,d\,) \) in which the
metric satisfies a strong form of the triangle inequality:
\[ d(x_{1}, x_{2}) \le \max \bigl\{ d(x_{1},x_{3}), d(x_{3},x_{2}) \bigr\} \]
for all \( x_{1}, x_{2}, x_{3} \in X \). A \emph{Polish space\/} is a separable completely metrizable topological space,
and a \emph{Polish group\/} is a topological group which is a Polish space.  By a \emph{metric group\/} we mean a
pair \( (G,d\,) \), where \( G \) is a topological group, \( d \) is a metric on \( G \), and the topology induced by
\( d \) coincides with the topology of \( G \); such metrics will be called \emph{compatible}.  A metric \( d \) on
\( G \) is said to be \emph{left invariant\/} if
\[ d(fg_{1}, fg_{2}) = d(g_{1},g_{2}) \]
for all \( f, g_{1}, g_{2} \in G \); the definition of a \emph{right invariant\/} metric is symmetric.  A metric \( d \)
on \( G \) is \emph{two-sided invariant\/} if it is both left and right invariant.

We also need the notion of a group completion.  If \( (G,d\,) \) is a metric group with a left invariant metric \( d \),
we let \( D \) be the metric on \( G \) defined by
\[ D(g_{1},g_{2}) = d(g_{1}, g_{2}) + d\bigl(g_{1}^{-1},\, g_{2}^{-1}\bigr). \]
Note that \( D \) is compatible with the topology of \( G \), but in general it is neither left nor right invariant.
Let \( (\clG, D) \) denote the Hausdorff completion of the metric space \( (G,D) \). It turns out that the group
operations on \( G \) admit a unique extension to \( \clG \), and the complete metric \( D \) turns \( \clG \) into a
topological group.  The group \( \clG \) is called the \emph{group completion\/} of \( G \).  As a topological group
\( \clG \) does not dependent on the choice of the compatible left invariant metric \( d \) on \( G \).

Given two (ultra)metric spaces \( (X,d_{X}) \) and \( (Y,d_{Y}) \) and a common subspace \( A = X \cap Y \) with
\[ d_{X}(a_{1}, a_{2}) = d_{Y}(a_{1}, a_{2}) \quad \forall a_{1}, a_{2} \in A, \]
we define the \emph{(ultra)metric amalgam of \( X \) and \( Y \) over \( A \)} to be the metric space \( (Z, d_{Z}) \),
\( Z = X \cup Y \), \( d_{Z} \) extends both \( d_{X} \) and \( d_{Y} \), and for \( x \in X \) and \( y \in Y \)
\begin{displaymath}
  d_{Z}(x, y) =
  \begin{cases}
    \inf\limits_{a \in A} \bigl( d_{X}(x, a) + d_{Y}(a, y) \bigr)  & \textrm{in the metric setting},\\
    \inf\limits_{a \in A} \max \bigl\{ d_{X}(x, a), d_{Y}(a, y) \bigr\} & \textrm{in the ultrametric setting}.\\
  \end{cases}
\end{displaymath}
Note that \( (Z,d_{Z}) \) is again an (ultra)metric space and that \( (X,d_{X}) \) and \( (Y,d_{Y}) \) are naturally
subspaces of \( Z \).  By taking isometric copies of spaces we can define the amalgamation of \( (X,d_{X}) \) and
\( (Y,d_{Y}) \) over \( (A,d_{A}) \) whenever we have two isometric inclusions \( \iota_{X} : A \to X \) and
\( \iota_{Y} : A \to Y \).

Two-sided invariant (ultra)metrics are characterized among left invariant (ultra)metrics by the following inequality.
\begin{proposition}
  \label{prop:chracteristic-tsi-inequality-ultrametric}
  Let \( d \) be a left invariant metric on a group \( G \).  The metric \( d \) is two-sided invariant if and only
  if
  \[ d(g_{1} \cdots g_{n}, f_{1} \cdots f_{n}) \le \sum\limits_{i = 1}^{n} d(g_{i}, f_{i}) \]
  for all \( g_{i}, f_{i} \in G \) and all \( n \in \mathbb{N} \).  If \( d \) is an ultrametric, then moreover
  \[ d(g_{1} \cdots g_{n}, f_{1} \cdots f_{n}) \le \max\limits_{1 \le i \le n}\{d(g_{i}, f_{i})\}. \]
\end{proposition}

With a left invariant pseudo-metric \( d \) on a group \( G \) we may associate a pseudo-norm
\( \norm{} : G \to \mathbb{R}^{+} \) defined by \( \norm{f} = d(f,e) \) and satisfying for all
\( g, g_{1}, g_{2} \in G \)
\begin{enumerate}[(i)]
\item \( \norm{g} \ge 0 \), \( \norm{e} = 0 \); \( \norm{g} > 0 \) for \( g \ne e \) if and only if \( d \) is a metric;
\item \( \norm{g} = \norm{g^{-1}} \);
\item\label{item:triangle-inequality} \( \norm{g_{1} g_{2}} \le \norm{g_{1}} + \norm{g_{2}} \);
\end{enumerate}
If \( d \) is an ultrametric, then item \eqref{item:triangle-inequality} becomes
\begin{enumerate}[(i)]
\item[\eqref{item:triangle-inequality}\( ' \)] \( \norm{g_{1} g_{2}} \le \max\{ \norm{g_{1}}, \norm{g_{2}} \} \).
\end{enumerate}
If \( d \) is two-sided invariant, then
\begin{enumerate}[(i)]
\setcounter{enumi}{3}
\item \( \norm{g g_{1} g^{-1}} = \norm{g_{1}} \).
\end{enumerate}
The correspondence between left invariant metrics and norms is bijective: \( d(g_{1},g_{2}) = \norm{g_{1}^{-1} g_{2}} \)
is a left invariant metric on \( G \) for any norm \( \norm{} \).

Expression \( [m,n] \) will denote the interval of natural numbers \( \{m, m+1, \ldots, n-1, n\} \).  For a set \( X \)
we let \( \word{X} \) to denote the set of nonempty words in the alphabet \( X \).  The length of a word
\( w \in \word{X} \) is denoted by \( |w| \) and \( w(i) \) denotes its \( i \)th letter.  If \( w \in \word{X} \) is a
word of length \( n \) and \( F \subseteq [1,n] \), \( F = \{j_{1}, \ldots, j_{m}\} \) with
\( j_{1} < j_{2} < \cdots < j_{m} \), then \( w(F) \) denotes the word \( w(j_{1})w(j_{2}) \cdots w(j_{m}) \).  The
minimal element of \( F \) is denoted by \( m(F) \), and \( M(F) \) denotes its maximal element: \( m(F) = j_{1} \),
\( M(F) = j_{m} \).

\subsection{Graev (ultra)metrics on free groups.}
\label{sec:scal-gener-graev}

We now describe the construction of Graev metrics and Graev ultrametrics on free groups following \cite{MR2278689} and
\cite{MR3028617}.  Since these two constructions are very similar, we give them in parallel. A \emph{pointed
  (ultra)metric space\/} is a triple \( (X,d,e) \), where \( (X,d\,) \) is an (ultra)metric space and \( e \in X \) is a
distinguished point.  Let \( X^{-1} \) denote a copy of \( X \) with elements of \( X^{-1} \) being formal inverses of
the elements of \( X \) with the agreement \( X \cap X^{-1} = \{e\} \), that is, \( e^{-1} = e \).  Extend \( d \) to an
(ultra)metric on \( X^{-1} \) by declaring \( d\bigl( x^{-1},y^{-1} \bigr) = d(x,y) \) for all \( x, y \in X \).  Let
\( (\invx, d, e) \) denote the (ultra)metric amalgam of \( (X, d\,) \) and \( (X^{-1},d\,) \) over the subspace
\( \{ e \} = X \cap X^{-1} \), see Figure \ref{fig:invX}.  We can extend the inverse \( x \mapsto x^{-1} \) to a
function on \( \invx \) by setting \( (x^{-1})^{-1} = x \).  To summarize, starting from a pointed (ultra)metric space
\( (X,d,e) \) we construct in a canonical way a pointed (ultra)metric space \( (\invx, d, e) \), \( X \) is a subspace
of \( \invx \) and the function \( \invx \ni x \mapsto x^{-1} \) is an isometric involution.  We shall say that
\( \invx \) is obtained from \( X \) by \emph{adding formal inverses}.

\begin{figure}[htb]
  \includegraphics{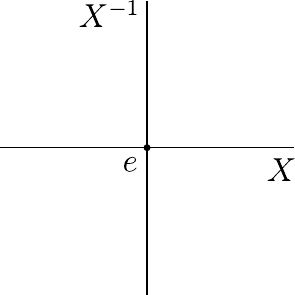}
  \caption{\( \skew8\overline{\rlap{\it X}\hbox to6.4pt{}} \)\, is the amalgam of \( X \) and \( X^{-1} \) over
    \( \{e\} \).}
  \label{fig:invX}
\end{figure}

\noindent By \( F(X) \) we denote the free group with generators \( X \setminus \{ e \} \) (therefore there is a slight
abuse of notations, since a proper notation would be \( F(X \setminus \{e\}) \)).  The set \( \invx \) is viewed as a
subset of \( F(X) \), where \( e \in \invx \) is identified with the identity element of the free group.  We have a
natural \emph{evaluation map\/} \( \word{\invx} \ni w \mapsto \hat{w} \in F(X) \), \( \hat{w} \) being just the
reduced form of \( w \).  This map is surjective.  For two words \( u_{1}, u_{2} \in \word{\invx} \) which have equal
lengths \( |u_{1}| = n = |u_{2}| \) we define
\begin{displaymath}
  \rho(u_{1}, u_{2})  =
  \begin{cases}
    \sum\limits_{i=1}^{n} d\bigl( u_{1}(i), u_{2}(i) \bigr) & \textrm{in the metric case},\\
    \max\limits_{i\le n} \bigl\{ d\bigl( u_{1}(i), u_{2}(i) \bigr) \bigr\} & \textrm{in the ultametric case}.
  \end{cases}
\end{displaymath}
Finally, the Graev (ultra)metric \( \delta \) on \( F(X) \) is defined by
\[ \delta(f_{1}, f_{2}) = \inf\bigl\{\, \rho(u_{1}, u_{2}) \,\bigm|\, u_{i} \in \word{\invx},\ \hat{u}_{i} = f_{i},\
|u_{1}| = |u_{2}| \,\bigr\}. \]

\begin{theorem}[Graev \cite{MR0038357}, Gao \cite{MR3028617}]
  \label{thm:graev-metric}
  The function \( \delta \) is a two-sided invariant (ultra)metric on the group \( F(X) \).  Moreover, \( \delta \)
  extends \( d \) on \( \invx \).
\end{theorem}

In order to describe an explicit formula for Graev metrics we need the notion of a match.  Let \( F \) be a finite set
of natural numbers \( F = \{i_{k}\}_{k=1}^{n} \), \( i_{1} < i_{2} < \cdots < i_{n} \).  A \emph{match\/} on \( F \) is a
bijection \( \theta : F \to F \) such that \( \theta \big( \theta (i) \big) = i \) for all \( i \in F \), and there are
no \( k < l \) such that \( \theta(i_{k}) = i_{p} \), \( \theta(i_{l}) = i_{q} \) and \( k < l < p < q \).  In other
words we can think of a match as a set of arcs connecting elements of \( F \) such that two arcs are either disjoint, or
one of the arcs is contained in the other one, see Figure \ref{fig:match}.  We shall sometimes say that \( \theta \) is
a match on a word \( w \) meaning that \( \theta \) is a match on its letters, which we then identify with the set
\( \{1, \ldots, |w| \} \).  If \( \theta \) is a match on \( F \) and \( \{j_{1}, \ldots, j_{m}\} \subseteq F \) such
that \( j_{1} < j_{2} < \cdots < j_{m} \), and \( \theta(\{j_{1}, \ldots, j_{m}\}) = \{j_{1}, \ldots, j_{m}\} \), then
the \emph{restriction\/} of \( \theta \) onto the set \( \{j_{1}, \ldots, j_{m}\} \) is a match on
\( \{j_{1}, \ldots, j_{m}\} \), but we shall abuse the terminology and say that \( \theta \) itself is a match on
\( \{j_{1}, \ldots, j_{m}\} \) in this case.

\begin{figure}[htb]
  \includegraphics{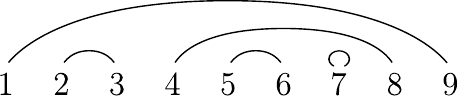}
  \caption{An example of a match on the set \( \{1, \ldots, 9\} \).}
  \label{fig:match}
\end{figure}

\noindent If \( w \in \word{\invx} \) and \( \theta \) is a match on \( w \), the word \( w^{\theta} \) is defined by
\begin{displaymath}
  w^{\theta}(i) =
  \begin{cases}
    w(i) & \textrm{if \( \theta(i) > i \)},\\
    e & \textrm{if \( \theta(i) = i \)},\\
    w\bigl( \theta(i) \bigr)^{-1} & \textrm{if \( \theta(i) < i \)}.\\
  \end{cases}
\end{displaymath}
Is is straightforward to check that \( \hat{w}^{\theta} = e \) for any \( w \) and any \( \theta \).  For example, if \(
\theta \) is the match in Figure \ref{fig:match}, then for the following word \( w \) we have
\begin{displaymath}
  \begin{array}{l c l@{\hspace{0.5cm}} l@{\hspace{0.5cm}} l@{\hspace{0.5cm}} l@{\hspace{0.5cm}} l@{\hspace{0.5cm}}
      l@{\hspace{0.5cm}} c@{\hspace{0.5cm}} l@{\hspace{0.5cm}} l}
    w & =        & x_{1} & x_{2} & x_{3}     & x_{4} & x_{5}& x_{6}     & x_{7} & x_{8} & x_{9}\\[2mm]
    w^{\theta} & = & x_{1} & x_{2} & x_{2}^{-1} & x_{4} & x_{5}& x_{5}^{-1} & e & x_{4}^{-1} & x_{1}^{-1}\\
  \end{array}
\end{displaymath}

\begin{theorem}[Sipacheva--Uspenskij \cite{MR913066}, Ding--Gao \cite{MR2332614}, Gao \cite{MR3028617}]
  \label{thm:computaion-of-Graev-metrics}
  Let \( \delta \) be the Graev (ultra)metric on the free group \( F(X) \).  For any \( f \in F(X) \)
  \[ \delta(f,e) = \min \bigl\{\, \rho\bigl(w,w^{\theta} \bigl) \ \bigm| \textrm{\( w \) is the reduced form of \( f \),
    \( \theta \) is a match on \( w \)} \,\bigr\}. \]
\end{theorem}

\medskip

In order to explain the generalized Graev metrics we need yet another tool.
\begin{definition}
  \label{def:scale}
  A \emph{scale\/} on a pointed set \( (X,e) \) is a function
  \( \Gamma : X\times \mathbb{R}^{+} \to \mathbb{R}^{+} \) satisfying for all \( x \in X \) and all
  \( r \in \mathbb{R}^{+} \)
  \begin{enumerate}[(i)]
  \item\label{item:smallest-at-identity} \( \Gamma(e,r) = r \), \( \Gamma(x,r) \ge r \);
  \item \( \Gamma(x,r) = 0 \) if and only if \( r = 0 \);
  \item\label{item:monotonicity} \( \Gamma(x, \cdot) \) is a monotone increasing function with respect to the second
    variable;
  \item \( \lim\limits_{r \to 0}\Gamma(x,r) = 0 \).
  \end{enumerate}
\end{definition}

By a \emph{scaled (ultra)metric space}, or for brevity just a \emph{scaled space}, we mean a tuple
\( (\invx, d, e, \Gamma) \), where \( \invx \) is obtained from some pointed (ultra)metric space \( X \) by adding
formal inverses and \( \Gamma \) is a scale on \( \invx \).  We shall denote scaled spaces with bold letters
\( \mathbf{X}\), \( \mathbf{Y} \), etc.  Let \( \mathbf{X} = (\invx, d, e, \Gamma) \) be a scaled space.  Following
\cite{MR2278689} and \cite{MR3028617}, for a match \( \theta \) on a word \( w \in \word{\invx} \) we define the number
\( \norm[\theta][\Gamma]{w} = \norm[\theta]{w} \) by induction on the length of \( w \) as follows.
\begin{enumerate}[(i)]
\item If \( w = x \) for some \( x \in \invx \), then \( \norm[\theta]{w} = d(x,e) \); if \( w = x_{1}x_{2} \) and
  \( \theta(x_{1}) = x_{2} \), then \( \norm[\theta]{w} = d(x_{1}, x_{2}^{-1}) \).
\item If \( \theta(x_{1}) = x_{k} \) and \( k < |w| \), then \( w = u_{1}u_{2} \) for some words
  \( u_{i} \in \word{\invx} \) with \( |u_{1}| = k \), \( |u_{2}| = n - k\), \( \theta \) is a match on both
  \( u_{1} \) and \( u_{2} \), and we set
  \begin{displaymath}
    \norm[\theta]{w} =
    \begin{cases}
      \norm[\theta]{u_{1}} + \norm[\theta]{u_{2}} & \textrm{in the metric case},\\
      \max\bigl\{ \norm[\theta]{u_{1}}, \norm[\theta]{u_{2}} \bigr\} & \textrm{in the ultrametric case}.\\
    \end{cases}
  \end{displaymath}
\item If \( \theta(x_{1}) = x_{n} \), \( n = |w| \), then let \( w = x_{1} u x_{n} \) for \( u \in \word{\invx} \),
  \( x_{1}, x_{n} \in \invx \), \( \theta \) is a match on \( u \) and we set
  \begin{displaymath}
    \norm[\theta]{w} =
    \begin{cases}
      d(x_{1},x_{n}^{-1}) + \min\Bigl\{\Gamma\Bigl(x_{1}^{-1}, \norm[\theta]{u} \Bigr),\
      \Gamma\Bigl(x_{n}, \norm[\theta]{u} \Bigr)\Bigr\} &  \textrm{in the metric case},\\
      \max\biggl\{ d\bigl( x_{1},x_{n}^{-1} \bigr), \min\Bigl \{\Gamma\Bigl( x_{1}^{-1}, \norm[\theta]{u} \Bigr),\
  \Gamma\Bigl(x_{n}, \norm[\theta]{u} \Bigr)\Bigr\} \biggr\} &  \textrm{in the ultrametric case}.\\
    \end{cases}
  \end{displaymath}
\end{enumerate}
\noindent The \emph{Graev (ultra)norm \( \norm[][\Gamma]{} = \norm{} \) of the scale \( \Gamma \)}  is
defined by
\begin{displaymath}
  \begin{aligned}
    \norm{f} &= \inf \bigl\{\, \norm[\theta]{w} \bigm| w \in \word{\invx},\ \hat{w} = f\ \textrm{and
      \( \theta \) is a
      match on \( w \)} \,\bigr\}.\\
  \end{aligned}
\end{displaymath}

\begin{proposition}[Ding--Gao~\cite{MR2278689}, Gao~\cite{MR3028617}]
  \label{thm:generalized-graev-metric}
  Let \( (X, d, e) \) be a pointed (ultra)metric space.  The function \( f \mapsto \norm{f} \) is an (ultra)norm on the
  group \( F(X) \), and the latter is a topological group in the topology of \( \norm{} \).  The natural inclusion map
  \( \invx \hookrightarrow F(X) \) is an isometry.
\end{proposition}

We denote by \( F(\mathbf{X}) \) the free group \( F(X) \) together with the Graev norm \( \norm{} \) and view
\( \invx \) as a subset of \( F(\mathbf{X}) \).

A \emph{canonical scale\/} on an (ultra)metric group \( (G,d\,) \) is a map
\( \cansc{G} : G \times \mathbb{R}^{+} \to \mathbb{R}^{+} \) defined by
\[ \cansc{G}(g,r) = \max \Bigl\{ r,\ \sup\bigl\{d(g^{-1} h g,e) \, \bigm|\, h \in G,\ d(h,e) \le r \bigr\}\Bigr\}.  \]
Let \( \mathbf{X} = (\invx, d, e, \Gamma) \) be a scaled space, \( (G,d_{G}) \) be an (ultra)metric group, and
\( \Gamma_{G} \) be a scale on \( G \).  A map \( \phi : \invx \to G \) is called a \emph{Lipschitz morphism with
  respect to the scale \( \Gamma_{G} \)} if for all \( x, y \in \invx \) and all \( r \in \mathbb{R}^{+} \)
\begin{enumerate}[(i)]
\item \( \phi(e) = e \);
\item \( \phi(x^{-1}) = \phi(x)^{-1} \);
\item \( d_{G}\big(\phi(x),\phi(y)\big) \le d(x,y) \);
\item \( \Gamma_{G}(\phi(x),r) \le \Gamma(x,r) \).
\end{enumerate}
We say that \( \phi \) is a Lipschitz morphism if it is a Lipschitz morphism with respect to the canonical
scale \( \cansc{G} \).  In general we shall use the term Lipschitz to mean \( 1 \)-Lipschitz.

\begin{proposition}[Ding--Gao \cite{MR2278689}, \cite{MR3028617}]
  \label{thm:exntesion-of-lipschitz-morphisms}
  Let \( \phi \) be a Lipschitz morphism from a scaled space \( \mathbf{X} \) into an (ultra)metric group \( G \).  The
  map \( \phi \) extends to a Lipschitz homomorphism \( \phi : F(\mathbf{X}) \to G \).  If \( G \) is completely
  metrizable, then \( \phi \) can be further extended to a continuous homomorphism
  \( \phi : \clF(\mathbf{X}) \to G \).
\end{proposition}

\section{Graev ultrametrics on free products}
\label{sec:graev-ultr-free}

Let \( (G,d_{G}) \) and \( (H,d_{H}) \) be ultrametric groups with two-sided invariant metrics, let \( A = G \cap H \)
be a common \emph{closed\/} subgroup; we assume that the metrics agree: \( d_{G}(a_{1}, a_{2}) = d_{H}(a_{1}, a_{2}) \)
for all \( a_{1}, a_{2} \in A \).  We shall define a two-sided invariant ultrametric \( \delta \) on the amalgamated
free product \( G *_{A} H \).

The construction of Graev metrics on free products mimics that on the free groups.  To start, we have a natural
\emph{evaluation map}:
\[ \word{G \cup H} \ni w \mapsto \hat{w} \in G*_{A}H, \]
where \( \hat{w} \) is just the product of letters of \( w \).  Note that this map is surjective.  If
\( w \in \word{G \cup H} \) and \( F \subseteq \bigl[1, |w|\bigr] \), then the evaluation of the subword \( w(F) \) is
denoted by \( \hat{w}(F) \) (as opposed to \( \widehat{w(F)} \)).  Let \( d \) be the ultrametric on the amalgam \( G
\cup H \) of \( (G, d_{G}) \) and \( (H, d_{H}) \) over \( A \).
If \( u_{1}, u_{2} \in \word{G \cup H} \) are two words of the same length
\( |u_{1}| = n = |u_{2}| \), we define \( \rho(u_{1}, u_{2}) \) to be the maximum of distances between the corresponding
letters:
\[ \rho(u_{1}, u_{2}) = \max\limits_{i \le n} \bigl\{ d\bigl( u_{1}(i), u_{2}(i) \bigr)\bigr\}. \]
The \emph{Graev ultrametric} on the free product \( G*_{A}H \) is the function
\[ \delta(f_{1}, f_{2}) = \inf \bigl\{\, \rho(u_{1}, u_{2}) \,\bigm|\, \hat{u}_{i} = f_{i},\ |u_{1}| = |u_{2}| \,\bigr\}. \]
Our goal is to prove
\begin{theorem}
  \label{thm:Graev-free-prodcut-ultrametric}
  The function \( \delta \) is a two-sided invariant ultrametric on \( G*_{A}H \).  Moreover, \( \delta \) extends \( d
  \) on \( G \cup H \).
\end{theorem}

As is typical for Graev metrics, it is straightforward to check that \( \delta \) is a two-sided invariant
pseudo-ultrametric.  The main difficulty is to show that distinct elements are never glued:
\( \delta(f_{1}, f_{2}) > 0 \) whenever \( f_{1} \ne f_{2} \).

Our arguments here are very similar to those in \cite{1111.1538}, and we shall outline the proofs and give references
for more details.  Essentially the proofs are repetitions of the proofs fore the metric case when the summation
operation is substituted with the operation of taking maximum, e.g., like in Proposition
\ref{prop:chracteristic-tsi-inequality-ultrametric}.  Another, more important difference is that in \cite{1111.1538} the
notion of an evaluation tree was used.  In our approach here we use instead the formalism of (maximal) evaluation
forests, which has the advantage of working for both the amalgams and HNN extensions in a uniform way.

The following proposition is essentially obvious.

\begin{proposition}[cf.~Lemma 5.1 \cite{1111.1538}]
  \label{prop:delta-is-pseudo-ultrametric}
  The function \( \delta \) is a two-sided invariant pseudo-ultrametric on \( G*_{A}H \).
\end{proposition}

Let \( \norm{} : G*_{A} H \to \mathbb{R}^{+} \) be the pseudo-norm that corresponds to \( \delta \):
\( \norm{f} = \delta(f,e) \).  In order to show that \( \delta(f_{1}, f_{2}) > 0 \) for \( f_{1} \ne f_{2} \) it is
enough to show that \( \norm{} \) is a genuine norm: \( \norm{f} > 0 \) for \( f \ne e \).

A pair of words \( (\alpha, \zeta) \), \( \alpha, \zeta \in \word{G \cup H} \), is said to be an \emph{\( f \)-pair\/} if
\( |\alpha| = |\zeta| \), \( \hat{\zeta} = e \), and \( \hat{\alpha} = f \).  The definition of the function \( \norm{} \)
can then be reformulated as
\[ \norm{f} = \inf \bigl\{\, \rho(\alpha, \zeta) \bigm| \textrm{\( (\alpha, \zeta) \) is an \( f \)-pair} \,\bigr\}. \]
To get a better understanding of the function \( \norm{} \), we shall gradually add restrictions on the \( f \)-pairs
\( (\alpha, \zeta) \), while still keeping the equality above.

\subsection{Trivial words}
\label{sec:trivial-words}

Before going any further we need to understand the structure of trivial words in the amalgam \( G*_{A}H \).  A word
\( \zeta \in \word{G \cup H} \) is said to be \emph{trivial\/} if \( \hat{\zeta} = e \).

We say that two letters \( x, y \in G \cup H \) are \emph{multipliable\/} if they both come from the same group: either
\( x, y \in G \) or \( x, y \in H \).  We also say that a word \( w \in \word{G \cup H} \) is \emph{multipliable\/} if
all of its letters come from the same group.

\begin{definition}
  \label{def:tree}
  Let \( (\mathcal{T}, \preceq) \) be a poset, and let \( s,t \in \mathcal{T} \).  We say that \( s \) is an
  \emph{immediate predecessor\/} of \( t \) if \( s \prec t \) and for any \( s \preceq s' \preceq t \) either
  \( s' = s \) or \( s' = t \).  If \( s \) is an immediate predecessor of \( t \), then \( t \) is also said to be an
  \emph{immediate successor\/} of \( s \).  A \emph{finite rooted tree}, or just a \emph{tree}, is a finite poset
  \( (\mathcal{T}, \preceq) \) with a distinguished element \( \vn \), called the \emph{root}, such that for any \( t
  \in \mathcal{T} \)
  \begin{itemize}
  \item \( t \preceq \vn \), i.e., the root \( \vn \) is the largest element;
  \item \( \{\, s \in \mathcal{T} \mid t \preceq s \,\}\) is linearly ordered.
  \end{itemize}
  There is a natural graph structure on \( \mathcal{T} \): we put an edge between \( s \) and \( t \) whenever \( s \)
  is an immediate predecessor of \( t \) or \( t \) is an immediate predecessor of \( s \).  With this assignment of
  edges \( \mathcal{T} \) is a rooted tree in the sense of the graph theory.  In a tree \( \mathcal{T} \), any element
  \( t \in \mathcal{T} \), except for the root, has a unique immediate successor, which we denote by \( t^{+} \).
  A \emph{leaf\/} in a tree is an element without predecessors.
  
  A \emph{finite forest\/} is a finite poset \( (\mathcal{F}, \preceq) \) which is a disjoint union of rooted trees
  \( \mathcal{F} = \sqcup_{i=1}^{r} \mathcal{T}_{i} \), where two elements \( t, s \in \mathcal{F} \) are comparable if
  and only if they belong to the same tree.  The root of the tree \( \mathcal{T} \) is denoted by \( \vn(\mathcal{T})
  \) and \( \roots{\mathcal{F}} \) denotes the set of
  roots of trees in \( \mathcal{F} \):
  \[ \roots{\mathcal{F}} = \bigl\{\, \vn(\mathcal{T}) \bigm| \textrm{\( \mathcal{T} \) is a tree in \( \mathcal{F} \)}
  \, \bigr\}. \]

  An \emph{evaluation forest\/} on an interval \( [1,n] \) is a forest \( \mathcal{F} \) together with an assignment \(
  t \mapsto I_{t} \subseteq [1,n] \) such that for all \( t, s \in \mathcal{F} \)
  \begin{enumerate}[(i)]
  \item \( I_{t} \) is a non-empty subinterval of \( [1,n] \);
  \item \( [1,n] = \sqcup_{\mathcal{T} }\, I_{\vn(\mathcal{T})} \), where the union is taken over all trees \( \mathcal{T}
    \) in \( \mathcal{F} \);
  \item \( I_{s} \cap I_{t} \ne \vn \) if and only if \( s \) and \( t \) are comparable in \( \mathcal{F} \);
  \item \( s \preceq t \) if and only if \( I_{s} \subseteq I_{t} \);
  \item\label{item:strict-containment} if \( s \prec t \), then \( m(I_{t}) < m(I_{s}) \le M(I_{s}) < M(I_{t}) \), and in
    particular \( I_{s} \subset I_{t} \);
  \end{enumerate}

  Let \( \zeta \in \word{G \cup H} \) be a word of length \( n \) such that \( \hat{\zeta} \in A \) and let
  \( \mathcal{F} \) be an evaluation forest on \( [1,n] \).  We say that \( \mathcal{F} \) is an \emph{evaluation forest
    for \( \zeta \)} if additionally for all \( t \in \mathcal{F} \)
  \begin{enumerate}[(i)]
  \setcounter{enumi}{5}
  \item \( \hat{\zeta}(I_{t}) \in A \);
  \item \( \zeta(R_{t}) \) is multipliable, where \( R_{t} = I_{t} \setminus \bigcup_{s \prec t} I_{s} \); the set
    \( R_{t} \) is called the \emph{reminder\/} of the interval \( I_{t} \).
  \end{enumerate}

  We say that an interval \( I \subseteq [1,n] \) is \emph{decomposable\/} (relative to \( \zeta \)) if one can write
  \( I \) as a disjoint union of non-trivial subintervals \( I = J_{1} \sqcup J_{2} \) with \( \hat{\zeta}(J_{1}) \in A \)
  and \( \hat{\zeta}(J_{2}) \in A \); otherwise we say that \( I \) is \emph{indecomposable\/} (relative to \( \zeta \)).

  Let \( \mathcal{F} \) be an evaluation forest for \( \zeta \).  We say that \( \mathcal{F} \) is \emph{maximal} if
  for all \( t \in \mathcal{F} \)
  \begin{enumerate}[(i)]
    \setcounter{enumi}{7}
  \item\label{item:max-indecomposability} \( I_{t} \) is indecomposable;
  \item\label{item:max-subintervals} if \( J \subset I_{t} \) is a non-empty subinterval with \( \hat{\zeta}(J) \in A \),
    and for all \( s \prec t \) either \( I_{s} \subseteq J \) or \( I_{s} \cap J = \vn \), then
    \( J \subseteq \cup_{s \prec t} I_{s} \).
  \end{enumerate}
\end{definition}

\begin{example}
  \label{ex:evaluation-forest}
  The notion of a maximal evaluation forest is quite technical, and we illustrate it on a concrete example.  Consider
  the word \( \zeta \in \word{G \cup H} \) given by
\begin{displaymath}
  \begin{array}{cccccccccccccccccccccc}
    \zeta & = & g_{1} & b & g_{2} & h_{1} & h_{2} & g_{3} & g_{4} & h_{3} & g_{5} & g_{6} & g_{7} & h_{4}
    & g_{8}  & g_{9} &  h_{5} & h_{6} & g_{10} & h_{7} & h_{8} & g_{11}
  \end{array}
\end{displaymath}
where \( g_{i} \in G \setminus A \), \( h_{j} \in H \setminus A \), and \( b \in A \).  Suppose also that the following
identities hold:
\begin{displaymath}
  \begin{array}{
      r@{\hspace{1cm}}
      r@{\hspace{1cm}}
      r@{\hspace{1cm}}
      r
    }
    g_{3}g_{4} = a_{1} & h_{5} h_{6} = a_{3} & h_{1} h_{2} a_{1} h_{3} a_{2} h_{4} = a_{5} & g_{1} b g_{2} a_{5} g_{8} =
    a_{7}\\
    g_{5}g_{6}g_{7} = a_{2} & h_{7} h_{8} = a_{4} & g_{9} a_{3} g_{10} a_{4} g_{11} = a_{6} \\
  \end{array}
\end{displaymath}
for some \( a_{i} \in A \).  In particular, \( \hat{\zeta} = a_{7} \cdot a_{6} \in A \).  Pictorially cancellations in
\( \zeta \) can be represented as follows.
\begin{displaymath}
  \begin{array}{cccccccccccccccccccccc}
    &
    & \scriptstyle 1
    & \scriptstyle 2
    & \scriptstyle 3    
    & \scriptstyle 4
    & \scriptstyle 5
    & \scriptstyle 6
    & \scriptstyle 7
    & \scriptstyle 8
    & \scriptstyle 9
    & \scriptstyle 10
    & \scriptstyle 11
    & \scriptstyle 12
    & \scriptstyle 13 
    & \scriptstyle 14
    & \scriptstyle 15
    & \scriptstyle 16
    & \scriptstyle 17
    & \scriptstyle 18 
    & \scriptstyle 19
    & \scriptstyle 20 \\
    
    \zeta & = & g_{1} & b & g_{2} & h_{1} & h_{2} & g_{3} & g_{4} & h_{3} & g_{5} & g_{6} & g_{7} & h_{4} & g_{8} & g_{9} &
    h_{5} & h_{6} & g_{10} & h_{7} & h_{8} & g_{11} \\[-0.5\normalbaselineskip]
    
    &&&&&&
    & \multicolumn{2}{@{}l@{}}{%
      \rlap{%
        \hspace*{\arraycolsep}%
        \(\underbracket[0.4pt]{\hphantom{\mbox{\( g_{3} \)\( g_{4} \)\hspace*{ \dimexpr2 \arraycolsep }}}}\)%
      }
    }
    && \multicolumn{3}{@{}l@{}}{%
      \rlap{%
        \hspace*{\arraycolsep}%
        \(\underbracket[0.4pt]{\hphantom{\mbox{\( g_{5} \)\( g_{6} \)\( g_{7} \)\hspace*{ \dimexpr4 \arraycolsep }}}}\)%
      }
    }
    &&&& \multicolumn{2}{@{}l@{}}{%
      \rlap{%
        \hspace*{\arraycolsep}%
        \(\underbracket[0.4pt]{\hphantom{\mbox{\( h_{5} \)\( h_{6} \)\hspace*{ \dimexpr2 \arraycolsep }}}}\)%
      }
    }
    && \multicolumn{2}{@{}l@{}}{%
      \rlap{%
        \hspace*{\arraycolsep}%
        \(\underbracket[0.4pt]{\hphantom{\mbox{\( h_{7} \)\( h_{8} \)\hspace*{ \dimexpr2 \arraycolsep }}}}\)%
      }
    }\\[-0.5\normalbaselineskip]
    
    &&&&
    & \multicolumn{9}{@{}l@{}}{%
      \rlap{%
        \hspace*{\arraycolsep}%
        \(\underbracket[0.4pt]{\hphantom{\mbox{\( h_{1} \)\( h_{2} \)\( g_{3} \)\( g_{4} \)\( h_{3} \)\( g_{5} \)\( g_{6} \)%
              \( g_{7} \)\( h_{4} \)\hspace*{ \dimexpr16 \arraycolsep }}}} \)%
      }
    }
    &
    & \multicolumn{7}{@{}l}{%
      \rlap{%
        \hspace*{\arraycolsep}%
        \( \underbracket[0.4pt]{\hphantom{\mbox{\( g_{9} \)\( h_{5} \)\( h_{6} \)\( g_{10} \)\( h_{7} \)\( h_{8} \)%
              \( g_{11} \)\hspace*{ \dimexpr12 \arraycolsep }}}} \)%
      }
    } \\[-0.5\normalbaselineskip]
    
    && \multicolumn{7}{@{}l}{%
      \rlap{%
        \hspace*{\arraycolsep}%
        \( \underbracket[0.4pt]{\hphantom{\mbox{\( g_{1} \)\( e \)\( g_{2} \)\( h_{1} \)\( h_{2} \)\( g_{3} \)%
              \( g_{4} \)\( h_{3} \)\( g_{5} \)\( g_{6} \)\( g_{7} \)\( h_{4} \)\( g_{8} \)%
              \hspace*{ \dimexpr24 \arraycolsep }}}} \)%
      }
    }
  \end{array}
\end{displaymath}

The corresponding evaluation forest \( \mathcal{F}_{1} \) for \( \zeta \) is shown in Figure \ref{fig:evaluation-forest}.
But note that \( \zeta \) has other evaluation forests as well (for instance, the forest \( \mathcal{F}_{2} \) in Figure
\ref{fig:evaluation-forest}).

\begin{figure}[htb]
  \includegraphics{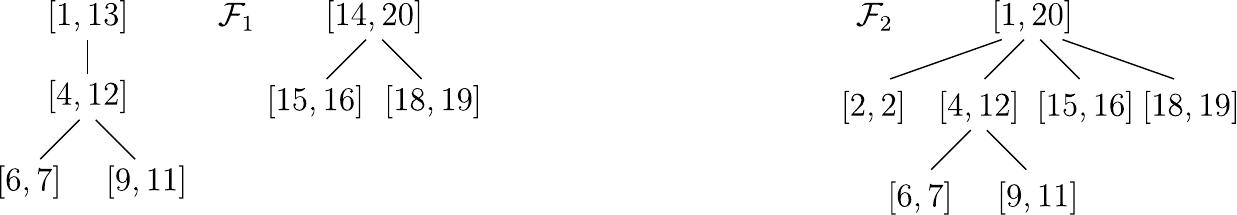}
  \caption{Two evaluation forests for \( \zeta \).  None of them is maximal.}
  \label{fig:evaluation-forest}
\end{figure}

Intuitively speaking an evaluation forest \( \mathcal{F} \) for \( \zeta \) captures the combinatorial structure of
cancellations: leaves of the forest are multipliable subwords that when multiplied produce an element from \( A \),
elements whose all predecessors are leaves correspond to subwords that after the evaluation of leaves become
multipliable and when multiplied yield and element from \( A \), etc.  Informally \( \mathcal{F} \) is a set of
subintervals with multipliable remainders such that two intervals are either disjoint or one is contained in the other,
and in the later case the containment is strict in the sense of item \eqref{item:strict-containment}.

The forests \( \mathcal{F}_{1} \) and \( \mathcal{F}_{2} \) in Figure \ref{fig:evaluation-forest} are not maximal.  In
\( \mathcal{F}_{1} \) item \eqref{item:max-subintervals} fails: for \( I = [1,13] \) we may add a subinterval
\( J = [2,2] \).  The forest \( \mathcal{F}_{2} \), which consists of a single tree, is not maximal because of the
failure of item \eqref{item:max-indecomposability}: we have \( I_{\vn} = [1,20] \) and
\( I_{\vn} = [1,13] \sqcup [14,20] \) with \( \hat{\zeta}\bigl( [1,13] \bigr) \in A \) and
\( \hat{\zeta}\bigl( [14,20] \bigr) \in A \).  But these are the only obstacles that prevent \( \mathcal{F}_{1} \) and
\( \mathcal{F}_{2} \) from being maximal (provided that no further relations between elements \( g_{i} \) and
\( h_{j} \) hold; for instance, if also \( h_{1} h_{2} \in A \), then the forest in Figure
\ref{fig:evaluation-forest-max} is not maximal either).  It is therefore easy to modify these forests to get a
maximal forest \( \mathcal{F}_{3} \), which is shown in Figure \ref{fig:evaluation-forest-max}.

\begin{figure}[htb]
  \includegraphics{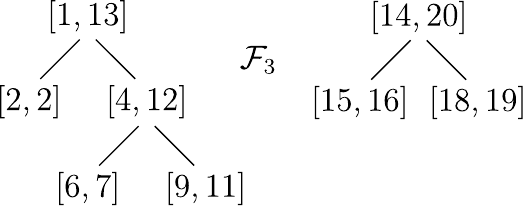}
  \caption{A maximal evaluation forest for \( \zeta \).}
  \label{fig:evaluation-forest-max}
\end{figure}

Item \eqref{item:max-subintervals} can be reformulated in a number of ways.  If \( J \subset I_{t} \) is a subinterval
such that \( J \cap I_{s} = \vn \) or \( I_{s} \subseteq J \) for any \( s \prec t \), then the condition
\( J \subseteq \bigcup_{s \prec t} I_{s} \) is equivalent to saying that \( J = \bigsqcup_{i = k}^{l} I_{s_{i}} \) for
some \( 1 \le k \le l \le m \), where \( s_{1}, \ldots, s_{m} \) are the immediate predecessors of \( t \) listed in the
order \( M(I_{s_{i}}) < m(I_{s_{i+1}}) \).  Yet another reformulation would be to say that
\( J \cap R_{t} = \vn \).  So item \eqref{item:max-subintervals} prohibits the situation shown in Figure
\ref{fig:item-max-subintervals-illustration}, where stars represent elements of \( R_{t} \).

\begin{figure}[htb]
  \includegraphics{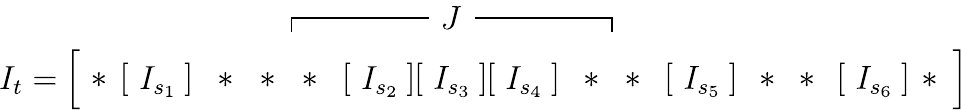}
  \caption{Such an interval \( J \) with \( \hat{\zeta}(J) \in A \) is prohibited by item \eqref{item:max-subintervals}.}
  \label{fig:item-max-subintervals-illustration}
\end{figure}
\end{example}

\begin{remark}
  \label{rem:elements-from-A-are-terminal-in-max-forests}
  If \( \mathcal{F} \) is a maximal forest on \( \zeta \) and \( i \) is such that \( \zeta(i) \in A \), then there must
  be a node \( t_{0} \in \mathcal{F} \) such that \( I_{t_{0}} = [i,i] \).  Indeed, if this were not the case, we would
  find the smallest \( t \in \mathcal{F} \) such that \( \zeta(i) \in I_{t} \) and obtain a contradiction with
  \eqref{item:max-subintervals} for \( I_{t} \) and \( J = [i,i] \).
\end{remark}

\begin{remark}
  \label{sec:non-uniqueness-of-maximal-forests}
  While there is much less freedom in constructing maximal evaluation forests when compared to general evaluation
  forests, there is still some amount of flexibility.  Here is a concrete example.  Let \( G = S_{6} \) --- the
  symmetric group on six elements, and let \( A = \{e\} \) be the trivial subgroup.  The group \( H \) does not matter,
  since our word will use only letters from \( G \).  Consider elements of \( G \)
  \begin{displaymath}
    g_{1} = (1 2), \quad g_{2} = (3 4), \quad g_{3} = (1 2)(3 4), \quad f_{1} = (1 2)(3 4)(5 6), \quad f_{2} = (5 6).
  \end{displaymath}
  The word
  \begin{displaymath}
    \begin{array}{cccccccc}
           & \scriptstyle  1  & \scriptstyle 2     & \scriptstyle 3 & \scriptstyle 4& \scriptstyle 5 & \scriptstyle 6 &
           \scriptstyle 7 \\
     \zeta = & f_{1} & g_{1} & g_{2} & g_{2} & g_{1} & g_{3} & f_{2} \\[1mm]
    \end{array}
  \end{displaymath}
    is trivial and it has two maximal evaluation forests (see Figure \ref{fig:evaluation-forest-max-non-uniqueness}).
  \begin{figure}[htb]
  \includegraphics{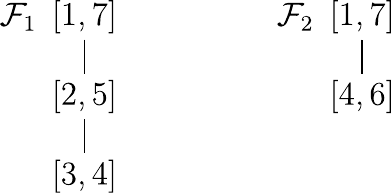}
  \caption{Two distinct maximal evaluation forests for \( \zeta \), each consisting of a single tree.}
  \label{fig:evaluation-forest-max-non-uniqueness}
\end{figure}
\end{remark}

\begin{proposition}
  \label{prop:existence-of-maximal-forest}
  Any word \( \zeta \in \word{G \cup H} \) with \( \hat{\zeta} \in A \) has a maximal evaluation forest.
\end{proposition}

\begin{proof}
  We prove the statement by induction on the length of \( \zeta \).  If \( |\zeta| = 1 \), then \( \zeta = a \) for some
  \( a \in A \), and therefore we may take \( \mathcal{F} \) to consist of a single root \( \mathcal{F} = \{ \vn \} \)
  with \( I_{\vn} = [1,1] \).

  Suppose the proposition has been proved for all words of length \( < n \) and let \( \zeta \) have length \( n \).

  \medskip

  \noindent \emph{Step 1: Decomposing \( [1,n] \).} If \( [1,n] \) is decomposable and \( [1,n] = J_{1} \sqcup J_{2} \)
  with \( \hat{\zeta}(J_{1}) \in A \) and \( \hat{\zeta}(J_{2}) \in A \), then we may apply the assumption of induction
  to the words \( \zeta_{1} = \zeta(J_{1}) \), \( \zeta_{2} = \zeta(J_{2}) \) and obtain their maximal evaluation
  forests \( \mathcal{F}_{1} \) and \( \mathcal{F}_{2} \) respectively.  The maximal evaluation forest for \( \zeta \)
  is then just the union of \( \mathcal{F}_{1} \) and \( \mathcal{F}_{2} \) with the natural assignment of intervals.

  \medskip

  We therefore may assume that \( [1,n] \) is indecomposable.
  
  \medskip

  \noindent \emph{Step 2: Typical case.} Suppose that we can find a proper subinterval \( J \subset [1,n] \) such that
  \( |J| \ge 2 \) and \( \hat{\zeta}(J) = a \in A \).  Let \( \zeta_{1} \) be the word obtained from \( \zeta \) by evaluating
  \( \zeta(J) \):
  \[ \zeta_{1} = \zeta\bigl( [1, m(J)-1] \bigr)\ a\ \zeta\bigl( [M(J)+1,n] \bigr). \]
  Since \( |J| \ge 2 \), the length of \( \zeta_{1} \) is less than \( n \), hence by the inductive assumption we may find
  an evaluation forest \( \mathcal{F}_{1} \) for \( \zeta_{1} \) with the assignment of intervals \( t \mapsto J_{t}
  \subseteq \bigl[1, |\zeta_{1}|\bigr] \).
  Let now \( I_{t} \) be the subintervals of \( [1,n] \) obtained from \( J_{t} \) by inserting \( \zeta(J) \) back into
  \( \zeta_{1} \), or, more formally:
  \begin{displaymath}
    I_{t} =
    \begin{cases}
      \bigl[ m(J_{t}), M(J_{t}) \bigr] & \textrm{if \( m(J_{t}), M(J_{t}) < m(J) \)},\\
      \bigl[ m(J_{t}), M(J_{t}) + |J| - 1 \bigr] & \textrm{if \( m(J_{t}) \le m(J) \le M(J_{t}) \)},\\
      \bigl[ m(J_{t}) + |J| - 1, M(J_{t}) + |J| -1 \bigr] & \textrm{if \( m(J) < m(J_{t}), M(J_{t})
      \)}.
    \end{cases}
  \end{displaymath}
  By the maximality of \( \mathcal{F}_{1} \) for \( \zeta_{1} \), the intervals \( J_{t} \) are indecomposable relative to
  \( \zeta_{1} \), but this may no longer be true for the intervals \( I_{t} \) relative to \( \zeta \), because an interval
  \( I_{t} \) with \( m(J) \in I_{t}\) has more possibilities for decomposition than the corresponding interval
  \( J_{t} \).

  The subword \( \zeta(J) \) itself has length \( < n \), and therefore the inductive assumption yields its maximal
  evaluation forest \( \widetilde{\mathcal{F}}_{2} \).  Since \( \zeta_{1}\bigl(m(J)\bigr) = a \in A \), by Remark
  \ref{rem:elements-from-A-are-terminal-in-max-forests} there is some \( t_{0} \in \mathcal{F}_{1} \) such that
  \( J_{t_{0}} = [m(J), m(J)] \).  The naive approach would be ``to put the forest \( \widetilde{\mathcal{F}}_{2} \)
  instead of the node \( t_{0} \)'' (see Figure \ref{fig:composing-forests-1}).  This does not work in general precisely
  because some of the intervals \( I_{t} \) may be decomposable.

  \begin{figure}[htb]
    \includegraphics{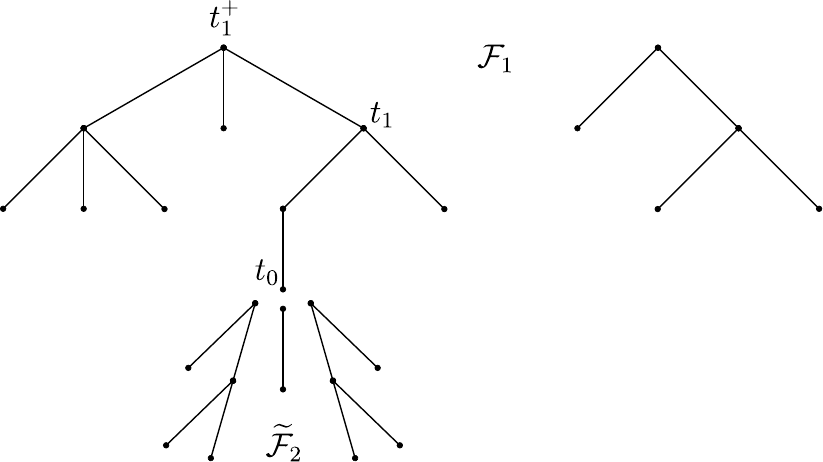}
    \caption{Naive approach of constructing \( \mathcal{F} \).}
    \label{fig:composing-forests-1}
  \end{figure}

  In order to fix this let \( t_{1} \in \mathcal{F}_{1} \) be the maximal node with \( I_{t_{1}} \) being decomposable
  (note that \( m(J) \in I_{t} \) for all decomposable intervals \( I_{t} \), hence such intervals are comparable, and
  the largest node \( t_{1} \) exists).  If all \( I_{t} \) are indecomposable, we set \( t_{1} = t_{0} \).
  Note that \( t_{1} \) is not the root of \( \mathcal{F}_{1} \), since
  \( I_{\vn} \) is assumed to be indecomposable.  In particular \( \zeta(I_{t_{1}}) \) has length strictly less than \( n \),
  and therefore by the assumption of induction it admits a maximal evaluation forest \( \mathcal{F}_{2} \) with
  intervals \( s \mapsto K_{s} \).

  \begin{figure}[htb]
    \includegraphics{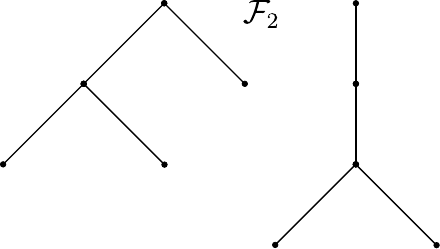}
    \caption{A possible example of the forest \( \mathcal{F}_{2} \).}
    \label{fig:composing-forests-2}
  \end{figure}

  We define the forest \( \mathcal{F} \) for \( \zeta \) by
  \[ \mathcal{F} = \{ s \in \mathcal{F}_{1} \mid s \not \preceq t_{1} \} \sqcup \mathcal{F}_{2}, \]
  with the ordering extending the orderings of \( \mathcal{F}_{1} \) and \( \mathcal{F}_{2} \) and
  \( s \prec t^{+}_{1} \) for all \( s \in \mathcal{F}_{2} \) (see Figure~\ref{fig:composing-forests-2} and
  Figure~\ref{fig:composing-forests-3}).  The assignment of intervals \( \mathcal{F} \ni t \mapsto I_{t} \) is the
  natural one: we have already defined \( I_{t} \) for \( t \in \mathcal{F}_{1} \cap \mathcal{F} \), and for
  \( s \in \mathcal{F}_{2} \) the interval \( I_{s} \) is just the interval \( K_{s} \) shifted by \( m(I_{t_{1}}) - 1 \):
  \[ I_{s} = \bigl[ m(K_{s}) + m(I_{t_{1}}) - 1, M(K_{s}) + m(I_{t_{1}}) -1 \bigr]. \]
  
  \begin{figure}[htb]
    \includegraphics{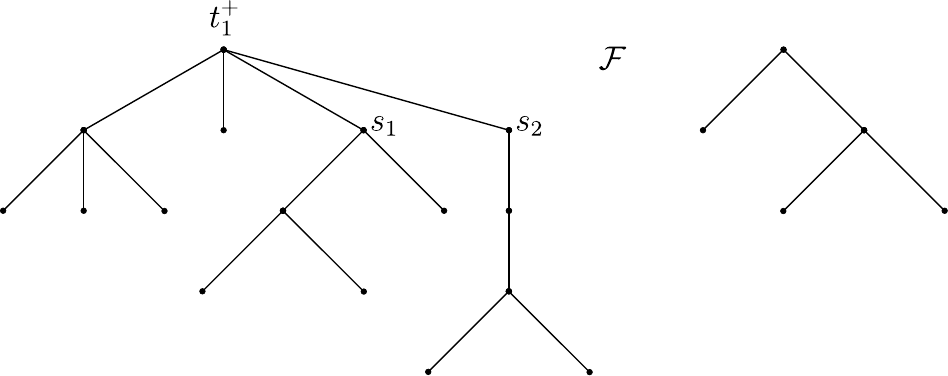}
    \caption{The forest
      \( \mathcal{F} = \{\, s \in \mathcal{F}_{1} \mid s \not \preceq t_{1} \,\} \sqcup \mathcal{F}_{2}\).}
    \label{fig:composing-forests-3}
  \end{figure}

  We claim that \( \mathcal{F} \) with \( s \mapsto I_{s} \) is a maximal evaluation forest for \( \zeta \).  It is
  straightforward to check that \( \mathcal{F} \) is an evaluation forest and item \eqref{item:max-indecomposability}
  follows immediately from the construction.  It remains to check item \eqref{item:max-subintervals}.  For
  \( s \ne t_{1}^{+} \) item \eqref{item:max-subintervals} follows from the maximality of \( \mathcal{F}_{1} \) and
  \( \mathcal{F}_{2} \), we need to check it only for \( t_{1}^{+} \).

  Suppose we have a subinterval \( L \subset I_{t_{1}^{+}} \) such that \( \hat{\zeta}(L) \in A \), for all
  \( s \prec t_{1}^{+} \) either \( I_{s} \cap L = \vn \) or \( I_{s} \subseteq L \) and
  \( L \cap R_{t_{1}^{+}} \ne \vn \).  Let \( s_{1}, \ldots, s_{m} \) be the immediate predecessors of \( t_{1}^{+} \),
  and let \( s_{k}, \ldots, s_{l} \) be those of the predecessors of \( t_{1}^{+} \) that correspond to the roots of
  \( \mathcal{F}_{2} \).  Note that the intervals \( I_{s_{i}} \) are adjacent for \( k \le i < l \):
  \( M(I_{s_{i}}) + 1 = m(I_{s_{i+1}}) \).  The idea is to construct an interval \( \widetilde{L} \) that will
  contradict item \eqref{item:max-subintervals} for \( t_{1}^{+} \) in the forest \( \mathcal{F}_{1} \). We have
  several cases.

  \emph{Case 1:  \( L \cap I_{s_{i}} = \vn \) for all \( k \le i \le l \).}  In this case the interval \( L \) naturally
  corresponds to a subword of~\( \zeta_{1} \); let
  \begin{displaymath}
    \widetilde{L} =
    \begin{cases}
      L & \textrm{if \( M(L) < m(I_{s_{k}}) \)},\\
      \bigl[ m(L) - |I_{t_{1}}| + 1, M(L) - |I_{t_{1}}| + 1 \bigr] & \textrm{if \( M(I_{s_{l}}) < m(L) \)}.
    \end{cases}
  \end{displaymath}
  We now get a contradiction with item \eqref{item:max-subintervals} of the maximality of \( \mathcal{F}_{1} \) for
  \( \zeta_{1} \) with \( J_{t_{1}^{+}} \) and \( \widetilde{L} \).

  \medskip

  We therefore may assume that either \( I_{s_{k}} \subseteq L \) or
  \( I_{s_{l}} \subseteq L \).  In either case, we may enlarge \( L \) to an interval \( L' \) defined by (see Figure
  \ref{fig:construction-of-L-prime})
  \[ L' = L \cup \Bigl(\, \bigsqcup_{i=k}^{l} I_{s_{i}} \Bigr), \]
  
  \begin{figure}[htb]
    \includegraphics{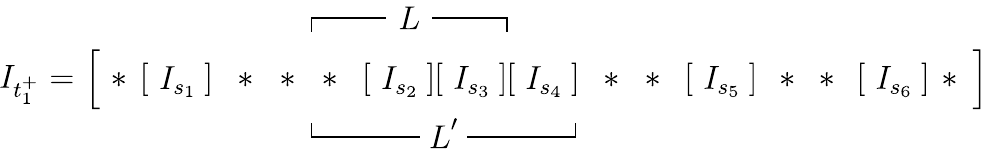}
    \caption{The construction of the interval \( L' \).  Here \( k = 2 \) and \( l = 4 \).}
    \label{fig:construction-of-L-prime}
  \end{figure}

  \emph{Case 2: \( L' = I_{t_{1}^{+}} \).}  In this case \( L \) is either an initial subinterval of \( I_{t_{1}^{+}}
  \), or a terminal subinterval.  In both cases \( I_{t_{1}}^{+} \) is decomposable contrary to the choice of
  \( t_{1} \).

  \emph{Case 3: \( L' \ne I_{t_{1}^{+}} \).} Since \( I_{s_{i}} \subset L' \) for all \( s_{i} \) that correspond to the
  roots of \( \mathcal{F}_{2} \), we may let \( \widetilde{L}' \) be the subinterval of \( J_{t_{1}^{+}} \) that
  corresponds to \( L' \):
  \[ \widetilde{L}' = \bigl[ m(L'), M(L') - |I_{t_{1}}| + 1 \bigl]. \]
  We again get a contradiction with item \eqref{item:max-subintervals} and maximality of \( \mathcal{F}_{1} \) for
  \( \zeta_{1} \), since \( \widetilde{L}' \) is a proper subinterval of \( J_{t_{1}^{+}} \) by the assumptions of this
  case.

  \medskip

  \noindent \emph{Step 3: Degenerate case.}  In the last step we suppose that one cannot find any \( J \subset [1,n] \)
  such that \( |J| \ge 2 \) and \( \hat{\zeta}(J) \in A \).  It is easy to see that in this case \( \zeta \) must be
  multipliable.  Let \( i_{1}, \ldots, i_{m} \) be the list of letters from \( A \): \( \zeta(i_{k}) \in A \) for all
  \( k \le m \).  We must have \( i_{k} + 1 < i_{k+1} \), because if we had two consecutive letters from \( A \) at
  indices, say, \( i \) and \( i+1 \), there would be a contradiction with the assumptions of this step for
  \( J = [i, i+1] \). (To be precise, we get a contradiction if also \( n \ge 3 \); if \( n = 2 \), and
  \( \zeta = a_{1} \, a_{2} \), then the maximal evaluation forest consists of two trivial trees \( [1,1] \) and
  \( [2,2] \).)  Note also that \( 1 < i_{1} \) and \( i_{m} < n \), because otherwise \( [1,n] \) is decomposable.  Put
  now \( \mathcal{F} = \{ \vn, t_{1}, \ldots, t_{m}\} \) and \( I_{\vn} = [1,n]\), \( I_{t_{k}} = [i_{k}, i_{k}] \).  It
  is straightforward to check that \( \mathcal{F} \) is a maximal forest for \( \zeta \).
\end{proof}

\subsection{Reductions}
\label{sec:reductions}

An \( f \)-pair \( (\alpha, \zeta) \) is said to be \emph{multipliable\/} if \( \alpha(i) \) is multipliable with
\( \zeta(i) \) for any \( i \). (We therefore use the word multipliable in two senses: a word is multipliable if all of
its letters come from the same group, while a pair of words is multipliable if for each index corresponding letters of
two words are from the same group.) Our first reduction states that in the definition of the norm function \( \norm{} \)
one may take only multipliable pairs.

\begin{lemma}[cf.~Lemma 5.2 \cite{1111.1538}]
  \label{lem:reduction-to-congruent-pairs}
  For any \( f \in G*_{A}H \)
  \[ \norm{f} = \inf \big\{\, \rho(\alpha, \zeta) \bigm| \textrm{\( (\alpha, \zeta) \) is a multipliable \( f \)-pair}
  \,\big\}. \]
\end{lemma}

\begin{proof}
  The idea of the proof is simple.  Let \( (\alpha, \zeta) \) be an \( f \)-pair.  Fix an \( \epsilon > 0 \).  If we have
  letters \( \alpha(i) \) and \( \zeta(i) \) which are not multipliable, then by the definition of the metric \( d \) on
  \( G \cup H \) we can find an element \( a \in A \) such that
  \[ d\bigl( \alpha(i), \zeta(i) \bigr) \ge \max \bigl\{ d\bigl( \alpha(i), a \bigr),\, d\bigl(a, \zeta(i) \bigr)\bigr\} -
  \epsilon. \]
  Let \( x = \alpha(i) \cdot a^{-1} \).  We now substitute the word `\( x \ a \)' into \( \alpha \) for the letter
  \( \alpha(i) \) and the word `\( e\ \zeta(i) \)' into \( \zeta \) for the letter \( \zeta(i) \).  In other words, if
  \( (\alpha, \zeta) \) is written as
  \begin{displaymath}
    \begin{array}{r@{\hspace{0.7cm}}ccccc}
                 &        & \scriptstyle i-1 &  \scriptstyle i & \scriptstyle i+1 &        \\[1mm]
        \alpha = & \cdots & \alpha(i-1)      &  \alpha(i)      & \alpha(i+1)      & \cdots \\[2mm]
        \zeta    = & \cdots & \zeta(i-1)         &  \zeta(i)         & \zeta(i+1)         & \cdots \\[2mm]
    \end{array}
  \end{displaymath}
  then the pair \( (\alpha_{1}, \zeta_{1}) \) after this substitution can be written as
  \begin{displaymath}
    \begin{array}{r@{\hspace{0.7cm}}cccccc}
      &        & \scriptstyle i-1 &  \scriptstyle i        & \scriptstyle i+1 & \scriptstyle i+2  & \\[1mm]
      \alpha_{1} = & \cdots & \alpha(i-1) &  \alpha(i) \cdot a^{-1} & a            & \alpha(i+1)       & \cdots \\[2mm]
      \zeta_{1}    = & \cdots & \zeta(i-1)    &  e                     & \zeta(i)       & \zeta(i+1)          &\cdots \\[2mm]
    \end{array}
  \end{displaymath}
  One now does this procedure for all \( i \) such that \( \alpha(i) \) and \( \zeta(i) \) are not multipliable.  The
  resulting pair \( (\beta, \xi) \) is multipliable, and by the two-sided invariance of the metric \( d \) we have
  \( \rho(\beta, \xi) \le \rho(\alpha, \zeta) + \epsilon \).  Since the pair \( (\alpha, \zeta) \) and \( \epsilon \) were
  arbitrary, we get
  \[ \norm{f} = \inf \bigl\{\, \rho(\beta, \xi) \bigm| \textrm{\( (\beta, \xi) \) is a multipliable \( f \)-pair}
  \,\bigr\}. \qedhere \]
\end{proof}

\begin{lemma}[cf.~Lemma 5.4 and Lemma 5.5 \cite{1111.1538}]
  \label{lem:existence-of-simple-pair-with-same-rho}
  Let \( (\alpha, \zeta) \) be a multipliable \( f \)-pair, and let \( \mathcal{F} \) be a maximal evaluation forest for
  \( \zeta \).  There exists a multipliable \( f \)-pair \( (\beta, \xi) \) such that
  \begin{enumerate}[(i)]
  \item \( |\xi| = |\zeta| \);
  \item \( \mathcal{F} \) is a maximal evaluation forest for \( \xi \) (with the same assignment \( t \mapsto I_{t}
    \));
  \item \( \rho(\alpha, \zeta) = \rho(\beta, \xi) \);
  \item\label{item:product-of-intervals-trivial} \( \hat{\xi}(I_{t}) = e \) for all \( t \in \mathcal{F} \);
  \end{enumerate}
\end{lemma}

\begin{proof}
  The proof is based on the following observation.  Let \( a \in A \), \( i < |\alpha| \), and define a pair
  \( (\alpha_{1}, \zeta_{1}) \) by changing \( \alpha(i) \) to \( \alpha(i) \cdot a^{-1} \), \( \alpha(i+1) \) to
  \( a \cdot \alpha(i+1) \) and also \( \zeta(i) \) to \( \zeta(i) \cdot a^{-1} \) and \( \zeta(i+1) \) to
  \( a \cdot \zeta(i+1) \):
  \begin{displaymath}
    \begin{array}{r@{\hspace{0.7cm}}cccc}
               &        &  \scriptstyle i & \scriptstyle i+1 & \\[1mm]
      \alpha = & \cdots &  \alpha(i) & \alpha(i+1) & \cdots \\[2mm]
      \zeta    = & \cdots &  \zeta(i)    & \zeta(i+1)    & \cdots \\[2mm]
      & & \multicolumn{2}{c}{\!\!\!\!\! \downarrow} &  \\[3mm]
     \alpha_{1} = & \cdots &  \alpha(i) \cdot a^{-1} & a \cdot \alpha(i+1) & \cdots \\[2mm]
     \zeta_{1} =    & \cdots &  \zeta(i) \cdot a^{-1}    & a \cdot \zeta(i+1)    & \cdots \\[2mm]
    \end{array}
  \end{displaymath}
  We call this operation a \emph{transfer operation}.  Observe that \( (\alpha_{1}, \zeta_{1}) \) is also a multipliable
  \( f \)-pair, \( \rho(\alpha, \zeta) = \rho(\alpha_{1}, \zeta_{1}) \), and \( \mathcal{F} \) is still a maximal
  evaluation forest for \( \zeta_{1} \), since \( \zeta_{1}(j) \in A \) if and only if \( \zeta(j) \in A \).  A typical
  application of the transfer is for \( a = \hat{\zeta}(I_{t}) \) and \( i = M(I_{t}) \), which yields a pair
  \( (\alpha_{1}, \zeta_{1}) \) with \( \hat{\zeta}_{1}(I_{t}) = e \).

  To obtain the desired pair \( (\beta, \xi) \), we apply the transfer operation for all intervals \( I_{t} \),
  \( t \in \mathcal{F} \setminus \roots{\mathcal{F}} \).  But this has to be done in a consistent order.  We traverse
  the forest \( \mathcal{F} \) ``from leaves to roots'' and ``from left to right''.  Somewhat more formally, we can
  define a function \( h : \mathcal{F} \to \mathbb{N}\) by \( h(t) \) being ``the longest downward path to a leaf''.
  For example, \( h(t) = 0 \) if and only if \( t \) is a leaf, \( h(t) = 1 \) if and only if all the predecessors of
  \( t \) are leaves, etc.

  Let \( t_{1}, \ldots, t_{m} \) be all the leaves of \( \mathcal{F} \) ordered in such a way that
  \( M(I_{t_{i}}) < m(I_{t_{i+1}}) \).  First we apply the transfer for \( a = \hat{\zeta}(I_{t_{1}}) \) at the index
  \( M(I_{t_{1}}) \) and obtain a pair \( (\alpha_{1}, \zeta_{1}) \) such that \( \hat{\zeta}_{1}(I_{t_{1}}) = e \); next we
  apply the transfer with \( a = \hat{\zeta}_{1}(I_{t_{2}}) \) to this new pair at the index \( M(I_{t_{2}}) \) and get
  \( (\alpha_{2}, \zeta_{2}) \) with \( \hat{\zeta_{2}}(I_{t_{1}}) = e \) and \( \hat{\zeta_{2}}(I_{t_{2}}) = e \), etc.  It
  is important that at the second step we take \( a = \hat{\zeta}_{1}(I_{t_{2}}) \) as opposed to
  \( a = \hat{\zeta}(I_{t_{2}}) \), since these may not be equal when \( M(I_{t_{1}}) + 1 = m(I_{t_{2}}) \).  Once we get
  \( (\alpha_{m}, \zeta_{m}) \), we continue with nodes \( t \) such that \( h(t) = 1 \), again ordering them ``from left
  to right.''

  From item \eqref{item:strict-containment} of the definition of the evaluation forest it follows that transfers at
  nodes with higher values of \( h(t) \) do not ruin the equalities \( \hat{\zeta}_{k}(I_{s}) = e \) for nodes \( s \)
  with smaller \( h(s) \).

  Note that transfer operations within different trees commute with each other.  We continue the above process for all
  \( t \in \mathcal{F} \setminus \roots{\mathcal{F}} \), and let \( (\beta_{1}, \xi_{1}) \) be the resulting pair.  It
  satisfies \( \rho(\beta_{1}, \xi_{1}) = \rho(\alpha, \zeta) \), \( \mathcal{F} \) is a maximal evaluation forest for
  \( \xi_{1} \), and \( \hat{\xi}_{1}(I_{t}) = e \) for all \( t \not \in \roots{\mathcal{F}} \).  To achieve the
  latter equality for roots, we again apply the transfer.  Let \( \vn_{1}, \ldots, \vn_{p} \) be the list of roots of
  \( \mathcal{F} \).  As usually we assume that \( M(I_{\vn_{i}}) < m(I_{\vn_{i+1}}) \).  We let
  \( (\beta_{2}, \xi_{2}) \) be the transfer of \( (\beta_{1}, \xi_{1}) \) with
  \( a = \hat{\xi_{1}}(I_{\vn_{1}}) \) at \( i = M(I_{\vn_{1}}) \); let \( (\beta_{3}, \xi_{3}) \) be the
  transfer of \( (\beta_{2}, \xi_{2}) \) with \( a = \hat{\xi_{2}}(I_{\vn_{2}}) \) at \( i = M(I_{\vn_{2}}) \); etc.
  We continue this process until the penultimate root \( \vn_{p-1} \): the pair \( (\beta_{p}, \xi_{p}) \) is obtained
  from \( (\beta_{p-1}, \xi_{p-1}) \) by transfer with \( a = \hat{\xi}_{p-1}(I_{\vn_{p-1}}) \) at
  \( i = M(I_{\vn_{p-1}}) \).
  
  We set \( (\beta, \xi) \) to be the pair \( (\beta_{p}, \xi_{p}) \) and claim that it satisfies the conclusion
  of the lemma.  All the items follow immediately from the construction with one exception: we have to explain why is it
  the case that \( \hat{\xi}(I_{\vn_{p}}) = e \).  This follows from the observation that
  \[ \hat{\xi} = \hat{\xi}(I_{\vn_{1}}) \cdot \hat{\xi}(I_{\vn_{2}}) \cdots \hat{\xi}(I_{\vn_{p-1}}) \cdot
  \hat{\xi}(I_{\vn_{p}}) = \hat{\xi}(I_{\vn_{p}}), \]
  and from \( \hat{\xi} = e \), since \( \hat{\xi} = \hat{\zeta} = e \).
\end{proof}

\begin{remark}
  \label{rem:simplicity-of-pairs}
  In the context of the above lemma it follows that \( \xi(i) = e \) whenever \( \xi(i) \in A \), since \(
  \mathcal{F} \) is maximal and any \( i \) with \( \xi(i) \in A \) corresponds to an interval \( I_{t} = [i,i] \) for
  some \( t \in \mathcal{F} \) by Remark \ref{rem:elements-from-A-are-terminal-in-max-forests}.
\end{remark}

An \( f \)-pair \( (\alpha, \zeta) \) with a maximal evaluation forest \( \mathcal{F} \) for \( \zeta \) is said to be
\emph{simple\/} if for all \( t \in \mathcal{F} \) one has \( \hat{\zeta}(I_{t}) = e \).  Lemma
\ref{lem:existence-of-simple-pair-with-same-rho} then implies that for any \( f \in G*_{A} H \)
\[ \norm{f} = \inf \big\{\, \rho(\alpha, \zeta) \bigm| \textrm{\( (\alpha, \zeta) \) is a simple \( f \)-pair}
  \,\big\}. \]

\subsection{Symmetrization}
\label{sec:symmetrization}

Simple pairs are important, because they allow for the following symmetrization operation.  Let \( (\alpha,\zeta) \) be a
multipliable \( f \)-pair with an evaluation forest \( \mathcal{F} \) and let \( t \in \mathcal{F} \) be a node with
the reminder \( R_{t} \).  Let \( i_{1} < i_{2} < \cdots < i_{m} \) be some of the elements of this reminder \( i_{k}
\in R_{t} \) and suppose that:
\begin{itemize}
\item \( \zeta(i_{k}) \not \in A \) for all \( k \);
\item \( \zeta(j) = e \) for all \( j \in R_{t} \setminus \{i_{k}\}_{k=1}^{m} \);
\item \( \hat{\zeta}(I_{s}) = e \) for all immediate predecessors \( s \prec t\);
\item \( \hat{\zeta}(I_{t}) = e \).
\end{itemize}
A typical example of such a situation comes from a simple pair \( (\alpha, \zeta) \) with a maximal evaluation forest
\( \mathcal{F} \): for some \( t \in \mathcal{F} \) with \( |I_{t}| \ge 2 \) we may set
\( \{ i_{k} \}_{k=1}^{m} = R_{t} \).  Under these assumptions the symmetrization of \( (\alpha, \zeta) \) with respect
to \( \{ i_{k} \}_{k=1}^{m} \) and \( k_{0} \), \( 1 \le k_{0} \le m \), is the pair \( (\alpha, \xi) \), where
\( \xi \) is defined by
\begin{displaymath}
  \xi(i) =
  \begin{cases}
    \zeta(i) & \textrm{ if \( i \ne i_{k} \) for all \( k \)},\\
    \alpha(i) & \textrm{ if \( i = i_{k} \) for \( k \ne k_{0} \)},\\
    \alpha(i_{k_{0} - 1})^{-1} \cdots \alpha(i_{1})^{-1} \cdot \alpha(i_{m})^{-1} \cdots \alpha(i_{k_{0}+1})^{-1} &
    \textrm{ if \( i = i_{k_{0}} \)}.
  \end{cases}
\end{displaymath}
Schematically symmetrization is showed on the following diagram:
\begin{displaymath}
  \begin{array}{r@{\hspace{0.7cm}}ccccccccccccc}
           &  & \scriptstyle i_{1} &  & \scriptstyle i_{2} & & \scriptstyle i_{k_{0}-1} &  & \scriptstyle i_{k_{0}} & &
           \scriptstyle i_{k_{0}+1} & & \scriptstyle i_{n} & \\
  \alpha = & \cdots & g_{1} & \cdots & g_{2} & \cdots \cdots & g_{k_{0}-1} & \cdots & g_{k_{0}} & \cdots & g_{k_{0}+1} & \cdots
  \cdots & g_{n} & \cdots \\[2mm]
  \zeta    = & \cdots & * & \cdots & * & \cdots \cdots & * & \cdots & * & \cdots & * & \cdots
  \cdots & * & \cdots \\[2mm]
  & & & & & & & \multicolumn{1}{c}{\downarrow} & & & & &  &  \\[2mm]
  \xi  = & \cdots & g_{1} & \cdots & g_{2} & \cdots \cdots & g_{k_{0}-1} & \cdots & x & \cdots & g_{k_{0}+1} & \cdots
  \cdots & g_{n} & \cdots 
\end{array}  
\end{displaymath}
where \( x \) is such that
\[ \hat{\xi}(I_{t}) = \hat{\xi}(R_{t}) = g_{1} \cdots g_{k_{0}-1} x g_{k_{0}+1} \cdots g_{m} = e,\]
i.e., \( x = g_{k_{0} - 1}^{-1} \cdots g_{1}^{-1} \cdot g_{m}^{-1} \cdots g_{k_{0}+1}^{-1} \).

If \( (\alpha, \zeta) \) is a multipliable pair, \( t \in \mathcal{F} \), and the list \( i_{1} < \cdots < i_{m} \) of
elements in \( R_{t} \) satisfies the requirements for symmetrization, we call such a list \emph{symmetrization
  admissible}.

\begin{lemma}[cf.~Lemma 5.6 \cite{1111.1538}]
  \label{lem:symmetrization-preserves-rho}
  If \( (\alpha, \zeta) \) is a multipliable \( f \)-pair with an evaluation forest \( \mathcal{F} \), and
  \( (\alpha, \xi) \) is obtained from \( (\alpha,\zeta) \) by symmetrization according to a symmetrization admissible
  list \( \{i_{k}\}_{k=1}^{m} \), then \( (\alpha, \xi) \) is also a multipliable \( f \)-pair, \( \mathcal{F} \) is an
  evaluation forest for \( \xi \) and \( \rho(\alpha, \xi) \le \rho(\beta, \zeta) \).
\end{lemma}

\begin{proof}\belowdisplayskip=-30pt
  The proof follows from the following calculations:
    \begin{align*}
      d \bigl(\alpha(i_{k_{0}}), x \bigr) &= d \bigl( \alpha(i_{k_{0}}), \alpha(i_{k_{0} - 1})^{-1} \cdots
      \alpha(i_{1})^{-1} \cdot \alpha(i_{m})^{-1} \cdots \alpha(i_{k_{0}+1})^{-1} \bigr) \\
      &= \hbox to 245.67276pt{\( d \bigl( \alpha(i_{1}) \cdots \alpha(i_{n}), e \bigr) \hfil \textrm{[\,by the two-sided
          invariance of \( d \)\,]} \)}\\
      &= \hbox to 245.67276pt{\( d \bigl( \alpha(i_{1}) \cdots \alpha(i_{n}), \zeta(i_{1}) \cdots \zeta(i_{n}) \bigr) \hfil
        \textrm{[\,since \( \hat{\zeta}(R_{t}) = e \)\,]} \)}\\
      &\le \hbox to 245.67276pt{\( \max\limits_{k \le n} d \bigl( \alpha(i_{k}), \zeta(i_{k}) \bigr) \hfil \textrm{[\,by
          Proposition \ref{prop:chracteristic-tsi-inequality-ultrametric}\,]}
        \)}\\
     \end{align*}\qedhere
\end{proof}

We say that a simple \( f \)-pair with a maximal evaluation forest \( (\alpha, \zeta) \) is \emph{reduced\/} if
\( \alpha \) is a reduced form of \( f \).  Note that when \( A \ne \{e\} \), the reduced form of an element is not
unique, but the length of the reduced form is nevertheless well-defined.

\begin{lemma}[cf.~Lemma 5.8 \cite{1111.1538}]
  \label{lem:norm--reduced-pairs}
  For any \( f \in G*_{A}H \)
  \[ \norm{f} = \inf \bigl\{\, \rho(\alpha, \zeta) \bigm| \textrm{\( (\alpha, \zeta) \) is a reduced \( f \)-pair}
  \,\bigr\}. \]  
\end{lemma}
\begin{proof}
  We start with an observation.  Let \( (\alpha, \beta) \) be a multipliable \( f \)-pair and suppose that there is an
  index \( i \) such that letters \( \alpha(i) \), \( \alpha(i+1) \), \( \zeta(i) \) and \( \zeta(i+1) \) are pairwise
  multipliable.  We may shorten the pair \( (\alpha, \zeta) \) by considering the products \( \alpha(i) \cdot \alpha(i+1)
  \) and \( \zeta(i) \cdot \zeta(i+1) \) as single letters.  More formally, we let the word \( \beta \) to be defined by
  \begin{displaymath}
    \beta(j) =
    \begin{cases}
      \alpha(j) & \textrm{if \( j < i \)},\\
      \alpha(i) \cdot \alpha(i+1) & \textrm{if \( j = i \)},\\
      \alpha(j+1) & \textrm{if \( j > i \)}.
    \end{cases}
  \end{displaymath}
  The word \( \xi \) is defined similarly using \( \zeta \) instead of \( \alpha \).  The pair
  \( (\beta, \xi) \) is also a multipliable \( f \)-pair, \( |\beta| < |\alpha| \), and
  \( \rho(\beta, \xi) \le \rho(\alpha, \zeta) \), since by Proposition
  \ref{prop:chracteristic-tsi-inequality-ultrametric}
  \[ d \bigl( \alpha(i) \cdot \alpha(i+1), \zeta(i) \cdot \zeta(i+1) \bigr) \le \max \bigl\{ d \bigl( \alpha(i), \zeta(i)
  \bigr), d \bigl(\alpha(i+1), \zeta(i+1) \bigr) \bigr\}.\]

  Note that a word \( \alpha \in \word{G \cup H} \) with \( \hat{\alpha} = f \) is a reduced form of \( f \) if and only
  if \( \alpha \) is the shortest word that evaluates to \( f \): if \( \alpha_{1} \in \word{G \cup H} \) is such that
  \( \hat{\alpha}_{1} = f \), then \( |\alpha_{1}| \ge |\alpha| \).  Based on this Lemma
  \ref{lem:existence-of-simple-pair-with-same-rho} implies that if \( (\alpha, \zeta) \) is a multipliable \( f \)-pair
  in which \( \alpha \) is a reduced form of \( f \), then there exists a \emph{simple\/} \( f \)-pair
  \( (\beta,\xi) \) such that \( \rho(\beta, \xi) \le \rho(\alpha, \zeta) \) and \( |\beta| = |\alpha| \), i.e.,
  \( (\beta, \xi) \) is a reduced \( f \)-pair.  Hence to prove the lemma it is enough to show that for any
  non-reduced simple \( f \)-pair \( (\alpha, \beta) \) there is a multipliable \( f \)-pair \( (\beta, \xi) \) such
  that \( \rho(\beta, \xi) \le \rho(\alpha, \zeta) \) and \( |\beta| < |\alpha| \).

  Pick a non-reduced simple \( f \)-pair \( (\alpha, \zeta) \).  If \( i \) such that \( \alpha(i) \) and
  \( \alpha(i+1) \) are multipliable and \( \alpha(i), \alpha(i+1) \not \in A \), then
  \( \alpha(i), \alpha(i+1), \zeta(i)\), and \( \zeta(i+1) \) are pairwise multipliable and we may shorten the pair by our
  observation above.  We therefore need to consider the case \( \alpha(i) \in A \) for some \( i \).  Let
  \( t \in \mathcal{F} \) be such that \( i \in R_{t} \).

  In a typical situation \( |R_{t}| \ge 2 \) and we may choose
  \( j \in R_{t} \) such that \( j \ne i \).  Let \( (\alpha, \xi) \) be the symmetrization of \( (\alpha,\zeta) \)
  according to \( R_{t} \) at \( j \).  By Lemma \ref{lem:symmetrization-preserves-rho}
  \( \rho(\alpha, \xi) \le \rho(\alpha, \zeta) \) and also \( \xi(i) = \alpha(i) \in A \).  Since
  \( (\alpha, \xi) \) is also a multipliable \( f \)-pair, all the elements \( \alpha(i) \), \( \alpha(i+1) \),
  \( \xi(i) \), and \( \xi(i+1) \) are pairwise multipliable, and we may finish the proof as before by
  shortening the pair \( (\alpha, \xi) \).

  Finally, if \( |R_{t}| = 1 \), then \( R_{t} = I_{t} = [i,i] \), hence \( \zeta(i) = e \), and again \( \alpha(i) \),
  \( \alpha(i+1) \), \( \zeta(i) \), and \( \zeta(i+1) \) must be pairwise multipliable.
\end{proof}

\begin{theorem}[cf.~Proposition 5.9 and Theorem 5.10 \cite{1111.1538}]
  \label{thm:Graev-ultrametric}
  The function \( \delta \) is a two-sided invariant ultrametric on \( G*_{A}H \).  Moreover, \( \delta \) extends \( d
  \) on \( G \cup H \).
\end{theorem}

\begin{proof}
  By Lemma \ref{lem:norm--reduced-pairs} we have
  \[ \norm{f} = \inf \bigl\{\, \rho(\alpha, \zeta) \bigm| \textrm{\( (\alpha, \zeta) \) is a reduced \( f \)-pair}
  \,\bigr\}. \]
  First we show that \( \delta \) extends \( d \).  If \( f \in G \cup H \), then the unique reduced \( f
  \)-pair is the pair \( (f,e) \), whence
  \[ \delta(f,e) = \norm{f} = d(f,e). \]
  If \( g \in G \) and \( h \in H \), then reduced \( gh^{-1} \)-pairs are of the form
  \( (g_{1} \ h_{1}^{-1}, e \ e) \), where \( g_{1} = g \cdot a \) and \( h_{1}^{-1} = a^{-1} \cdot h^{-1} \) for some
  \( a \in A \).  Therefore
  \[ d(g_{1}, e) = d\bigl( g, a^{-1} \bigr), \quad d\bigl( h_{1}^{-1}, e \bigr) = d\bigl(h^{-1}, a \bigl) =
  d\bigl(a^{-1}, h\bigl). \]
  Since \( d \) by definition is the ultrametric amalgam of metrics \( d_{G} \) and \( d_{H} \) on \( A \), it
  follows that \( \delta(g,h) = d(g,h) \).  Thus \( \delta \) extends \( d \) on \( G \cup H \).

  We show that \( \norm{f} > 0 \) for any \( f \ne e \).  Since we already know that \( \delta \) extends \( d \), it
  is enough to consider the case \( f \not \in A \).  Pick a reduced form \( \alpha_{0} \) of \( f \) and let
  \( \epsilon \) be such that \( d\bigl( \alpha_{0}(i), A \bigr) \ge \epsilon > 0 \) for all \( i \) (here we use that
  \( A \) is closed in both \( G \) and \( H \)).  Note that if \( \alpha \) is any other reduced form of
  \( f \), then \( |\alpha| = |\alpha_{0}| \) and \( A \alpha(i) A = A \alpha_{0}(i) A \) for all \( i \).

  So let \( (\alpha, \zeta) \) be any reduced \( f \)-pair, and let \( \mathcal{F} \) be an evaluation forest for
  \( \zeta \).  Pick a leaf \( t \in \mathcal{F} \).  The subword \( \zeta(I_{t}) \) is multipliable.  Since
  \( \alpha(i) \) is multipliable with \( \zeta(i) \) for all \( i \), and since \( \alpha(i) \) is \emph{not\/}
  multipliable with \( \alpha(i+1) \) (because \( \alpha \) is reduced), we get that either
  \( \zeta\bigl(m(I_{t})\bigr) \in A \), or \( \zeta\bigl(m(I_{t}) + 1\bigr) \in A \).  In any case, there is an index
  \( j \) such that \( \zeta(j) \in A \).  This shows that
  \begin{displaymath}
    \rho(\alpha, \zeta) = \max_{i \le |\alpha|} \bigl\{ d\bigl( \alpha(i), \zeta(i) \bigr) \bigr\} \ge d\bigl( \alpha(j),
    \zeta(j) \bigr) \ge
    d \bigl( \alpha(j), A \bigr)  = d\bigl( \alpha_{0}(j), A \bigr) \ge \epsilon > 0.
  \end{displaymath}
  And therefore also \( \norm{f} \ge \epsilon \).  This proves that \( \norm{} \) is a genuine ultranorm on \( G*_{A}H
  \).
\end{proof}

\begin{remark}
  \label{rem:free-products-of-infinitely-many-group}
  The above result is valid for any number of factors: if \( (G_{\lambda}, d_{\lambda})_{\lambda \in \Lambda} \) is a
  family of ultrametric groups with two-sided invariant ultrametrics \( d_{\lambda} \), \( A \) is a common closed
  subgroup of the groups \( G_{\lambda} \), metrics \( d_{\lambda} \) agree on \( A \), then one can define in a
  similar way a two-sided invariant Graev ultrametric \( \delta \) on the free product \( \Ast_{A} G_{\lambda} \) over
  all \( \lambda \in \Lambda \), which extends metrics \( d_{\lambda} \).
\end{remark}

\subsection{Graev ultrametrics on HNN extensions}
\label{sec:graev-ultrametrics-hnn}

Let \( G \) be a group, \( A, B < G \) be its subgroups, and \( \phi : A \to B \) be an isomorphism.  One way to
construction the HNN extension of \( (G,\phi) \) is as follows.  We start with free products \( G*\langle u \rangle \)
and \( G* \langle v \rangle \), where \( \langle u \rangle \) and \( \langle v \rangle \) are free groups on one
generator.  The map \( \phi \) gives rise to an isomorphism \( G*uAu^{-1} \to G*vBv^{-1} \).  Let
\( \widetilde{H} \) be the amalgam of the groups \( G*\langle u \rangle \) and \( G*\langle v \rangle \) over the
subgroups \( G*uAu^{-1} \) and \( G*vBv^{-1} \) (which are canonically isomorphic to \( \langle G, uAu^{-1} \rangle \)
and \( \langle G, vBv^{-1} \rangle \) respectively).  The HNN extension of \( (G,\phi) \) is the subgroups of
\( \widetilde{H} \) generated by \( G \) and the element \( v^{-1} u \), called the \emph{stable letter} of the HNN
extension.

Our goal is to carry this construction in the setting of ultrametric groups.  This is done exactly as in \cite[Section 8
and 9]{1111.1538} with substituting the max operation for the operation of summation.  We therefore only state the main
lemmas and give references for their proofs in the metric setting.

Let \( (G,d\,) \) be an ultrametric group with a two-sided invariant metric \( d \), let \( A \) be a closed subgroup of
\( G \).  We consider the free product \( G * \langle u \rangle \).  To overload notations, let \( d \) denote also the
natural metric on \( \langle u \rangle \): \( d(u^{m}, u^{n}) = |m - n| \).  The subgroup
\( \bigl\langle G, u A u^{-1} \bigr\rangle \) of the free product \( G * \langle u \rangle \) is isomorphic to
\( G * uAu^{-1} \), and therefore has two natural metrics: the Graev ultrametric on the free product \( G * uAu^{-1} \)
and the metric induced from the Graev ultrametric on \( G* \langle u \rangle \).  We show that these two metrics
coincide if and only if the subgroup \( A \) has diameter at most \( 1 \).

Let \( \delta \) be the Graev ultrametric on \( G * \langle u \rangle \), and let \( f \in G* uAu^{-1} \), which we view
as a subgroup of \( G * \langle u \rangle \).

\begin{definition}
  \label{def:hereditary-f-pair}
  A multipliable \( f \)-pair \( (\alpha, \zeta) \) is said to be \emph{hereditary\/} if \( \zeta(i) = \alpha(i) \)
  whenever \( \zeta(i) \in \langle u \rangle \setminus \{e\} \).
\end{definition}

\begin{lemma}[cf.~Lemma 8.2 \cite{1111.1538}]
  \label{lem:reduction-hereditary}
  For any multipliable \( f \)-pair \( (\alpha, \zeta) \) there exists a word \( \xi \) such that the pair
  \( (\alpha, \xi) \) is a hereditary \( f \)-pair and \( \rho(\alpha, \xi) \le \rho(\alpha, \zeta) \).
\end{lemma}

\begin{remark}
  \label{rem:structure-of-forests-for-hereditary-pairs}
  Let \( (\alpha, \zeta) \) be a hereditary pair, and let \( \mathcal{F} \) be a maximal evaluation forest for
  \( \zeta \).  If \( i \) is such that \( \zeta(i) \in \langle u \rangle \setminus \{ e \} \) and \( i \in R_{t} \),
  \( t \in \mathcal{F} \), then \( |R_{t}| = 2 \).
\end{remark}

\begin{definition}
  \label{def:regid-pair}
  A hereditary \( f \)-pair \( (\alpha, \zeta) \) is said to be \emph{rigid\/} if for each \( i \) the equality
  \( \alpha(i) = u^{\pm 1} \) implies \( \zeta(i) = \alpha(i) \).
\end{definition}

\begin{lemma}[cf.~Lemma 8.15 \cite{1111.1538}]
  \label{lem:reduction-rigid}
  Let \( (\alpha, \zeta) \) be a hereditary \( f \)-pair, and suppose that \( \alpha \) is the reduced form of \( f \).
  There exists a word \( \xi \) such that the pair \( (\alpha, \xi) \) is a rigid \( f \)-pair,
  \( \rho(\alpha, \xi) \le \rho(\alpha, \zeta) \), and moreover if
  \( \alpha(i) = u \), then \( \xi(i+1) \in A \).
\end{lemma}

Let \( \delta \) be the Graev ultrametric on \( G * \langle u \rangle \) and \( \underline{d} \) be the Graev
ultrametric on \( G* uAu^{-1} \).

\begin{theorem}[cf.~Theorem 8.16 \cite{1111.1538}]
  \label{thm:metrics-agree-iff-diam-le-1}
  Two metrics agree \( \underline{d} = \delta|_{G* u A u^{-1}} \) if and only if \( \diam(A) \le 1 \).
\end{theorem}

\begin{theorem}[cf.~Theorem 9.1 \cite{1111.1538}]
  \label{thm:graev-metrics-on-hnn-extensions}
  Let \( (G, d\,) \) be an ultrametric group with a two-sided invariant metric \( d \), \( A \) and \( B \) be closed
  subgroups of \( G \) and \( \phi : A \to B \) be a \( d \)-isometric isomorphism.  If \( \diam(A) \le K \), then there
  exists a two-sided invariant ultrametric \( \delta \) on the HNN extension \( H \) of \( (G, \phi) \) which extends \(
  d \) and such that \( \delta(t, e) = K \), where \( t \) is the stable letter of \( H \).
\end{theorem}

\section{Free products of Polish groups}
\label{sec:free-products-polish}

In this section we introduce and investigate a notion of a free product of Polish groups.  Our construction goes as
follows.  First we define unions of scaled spaces and argue that the union of scaled spaces \( \mathbf{X} \) and
\( \mathbf{Y} \) gives rise to a natural notion of the free product of the free groups \( F(\mathbf{X}) \) and
\( F(\mathbf{Y}) \).  Next using the surjective universality of groups \( \clF(\mathbf{X}) \)
we define free products of Polish groups as factors of \( \clF(\mathbf{X} \cup \mathbf{Y}) \).  Our
construction of a free product of Polish groups \( G \) and \( H \) is not canonical.  It takes for input two scaled
spaces \( \mathbf{X} \), \( \mathbf{Y} \), left invariant metrics \( d_{G} \) and \( d_{H} \) on \( G \) and \( H \)
respectively, and surjective Lipschitz morphisms \( \phi_{G} : X \to G \) and \( \phi_{H} : Y \to H \).  Universal
properties of our construction that are reminiscent of the universal properties for free products of
abstract groups are given in Proposition \ref{thm:properties-of-the-free-product}.

While the results of the previous section are companions of the corresponding earlier result in the metric setting, the
free product construction of this section is new for both metric and ultrametric cases.

\begin{definition}
  \label{def:amalgam-of-scaled-spaces}
  Given two scaled (ultra)metric spaces \( \mathbf{X} = \bigl( \invx, d_{X}, e, \Gamma_{X} \bigr) \) and
  \( \mathbf{Y} = \bigl( \invy, d_{Y}, e, \Gamma_{Y} \bigr) \) we define their union
  \( \mathbf{X} \cup \mathbf{Y} = (\invz, d, e, \Gamma) \) to be the (ultra)metric amalgam of \( \invx \) and
  \( \invy \) over \( \{e\} \) (see Figure \ref{fig:union-scaled}).  More precisely, if \( Z \) is the (ultra)metric
  amalgam of the (ultra)metric spaces \( X \) and \( Y \) over the subspace \( \{e\} \), then as an (ultra)metric space
  \( \mathbf{X} \cup \mathbf{Y} \) is obtained from \( Z \) by adding formal inverses.  Note that \( \invz \) is also
  the amalgam of \( \invx \) and \( \invy \) over \( \{e\} \); in other words \( \invxcupy = \invx \cup \invy \).  The
  scale \( \Gamma \) on \( \invz \) is the union of scales \( \Gamma_{X} \) and \( \Gamma_{Y} \):
  \begin{displaymath}
    \Gamma(z,r) = 
    \begin{cases}
      \Gamma_{X}(z,r) & \textrm{if \( z \in \invx \)},\\
      \Gamma_{Y}(z,r) & \textrm{if \( z \in \invy \)}.\\
    \end{cases}
  \end{displaymath}
  \begin{figure}[htb]
    \centering
    \includegraphics{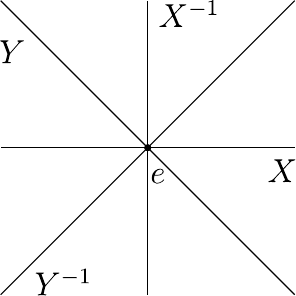}
    \caption{Union of scaled spaces.}
    \label{fig:union-scaled}
  \end{figure}
\end{definition}

Let \( \mathbf{X} \) and \( \mathbf{Y} \) be scaled (ultra)metric spaces and let
\( \pi_{X} : \invxcupy \to \invx \) be the retract map:
\begin{displaymath}
  \pi_{X}(z) = 
  \begin{cases}
    z & \textrm{if \( z \in \invx \)},\\
    e & \textrm{if \( z \in \invy \)}.
  \end{cases}
\end{displaymath}
This map is Lipschitz and it extends to a surjective group homomorphism
\( \pi_{X} : F(\mathbf{X} \cup \mathbf{Y}) \to F(\mathbf{X}) \).

\begin{proposition}
  \label{prop:existence-of-retraction}
  The homomorphism  \( \pi_{X} : F(\mathbf{X} \cup \mathbf{Y}) \to F(\mathbf{X}) \) is Lipschitz.
\end{proposition}

\begin{proof}
  Let \( f \in F(\mathbf{X} \cup \mathbf{Y}) \), \( \norm[][X]{} \) be the Graev (ultra)norm on
  \( F(\mathbf{X}) \), and \( \norm{} \) be the Graev (ultra)norm on \( F(\mathbf{X} \cup \mathbf{Y}) \).  Pick
  \( w \in \word{\invxcupy} \), \( \hat{w} = f \), and a match \( \theta \) on \( w \).  We need to show that
  \( \norm[][X]{} \bigl( \pi_{X}(f) \bigr) \le \norm{f} \), and for this it is enough to find a word
  \( u \in \word{\invx} \) and a match \( \mu \) on \( u \) such that \( \hat{u} = \pi_{X}(f) \) and
  \( \norm[\mu]{u} \le \norm[\theta]{w} \) (note that \( \norm[\mu][X]{u} = \norm[\mu]{u} \) because \( \mathbf{X} \) is
  a subspace of \( \mathbf{X} \cup \mathbf{Y} \), and therefore we omit the subscript).  If
  \( w = z_{1} \cdots z_{n} \) with \( z_{i} \in \invxcupy \), set \( u = \tilde{z}_{1} \cdots \tilde{z}_{n} \)
  with \( \tilde{z}_{i} = \pi_{X}(z_{i}) \).  We can view \( \theta \) as being also a match on \( u \).  Since
  \( \pi_{X} : \invxcupy \to \invx \) is Lipschitz, we have
  \( d\big( z_{i},z_{j}^{-1} \big) \ge d_{X} \big( \tilde{z}_{i}, \tilde{z}_{j}^{-1} \big) \) for all \( i,j \).  By
  item \eqref{item:smallest-at-identity} of the definition of the scale,
  \( \Gamma(z,r) \ge \Gamma\bigl( \pi_{X}(z),r \bigr) \) for all \( r \in \mathbb{R}^{+} \) and all
  \( z \in \invxcupy \).  It now follows from item \eqref{item:monotonicity} of the scale and from the definition
  of the norm that \( \norm[\theta]{u} \le \norm[\theta]{w} \).
\end{proof}

The homomorphism \( \pi_{X} \) is therefore continuous and extends to a continuous homomorphism
\[ \pi_{X} : \clF(\mathbf{X} \cup \mathbf{Y}) \to \clF(\mathbf{X}). \]
Note that \( \pi_{X}(f) = f \) for any \( f \in \clF(\mathbf{X}) \) and that \( \pi_{X}(f) = e \) holds true for
all \( f \in \clF(\mathbf{Y}) \).

\begin{corollary}
  \label{thm:embedding-is-isometric}
  Let \( \mathbf{X} \) and \( \mathbf{Y} \) be scaled (ultra)metric spaces.  The inclusion
  \( F(\mathbf{X}) \hookrightarrow F(\mathbf{X} \cup \mathbf{Y}) \) is isometric.
\end{corollary}

\begin{proof}
  Let \( \norm{} \) denote the Graev (ultra)norm on \( F(\mathbf{X} \cup \mathbf{Y}) \), and \( \norm[][X]{} \)
  be the Graev (ultra)norm on \( F(\mathbf{X}) \).  We need to show that for all \( f \in F(\mathbf{X}) \) one has
  \( \norm{f} = \norm[][X]{f} \).  By definition
  \begin{displaymath}
    \begin{aligned}
      \norm[][X]{f} &= \inf \Bigl\{\norm[\theta]{w} : w \in \word{\invx},\ \hat{w} = f\ \textrm{and \( \theta \)
        is
        a match on \( w \)}\Bigr\},\\
      \norm{f} &= \inf \Bigl\{\norm[\theta]{w} : w \in \word{\invxcupy},\ \hat{w} = f\ \textrm{and
        \( \theta \) is
        a match on \( w \)}\Bigr\},\\
    \end{aligned}
  \end{displaymath}
  and therefore \( \norm{f} \le \norm[][X]{f} \).  The reverse inequality follows immediately from \( \pi_{X}(f)
  = f \) for \( f \in F(\mathbf{X}) \) and Proposition \ref{prop:existence-of-retraction}.
\end{proof}

\begin{corollary}
  \label{cor:embedding-of-completions}
  Inclusion \( F(\mathbf{X}) \hookrightarrow F(\mathbf{X}\cup\mathbf{Y}) \) extends to
  \( \clF(\mathbf{X}) \hookrightarrow \clF(\mathbf{X} \cup \mathbf{Y}) \).
\end{corollary}
It is natural to regard \( \clF(\mathbf{X} \cup \mathbf{Y}) \) as being the free product of groups
\( \clF(\mathbf{X}) \) and \( \clF(\mathbf{Y}) \).

\bigskip

Let \( (G, d_{G}) \) and \( (H, d_{H}) \) be Polish groups with compatible left invariant (ultra)metrics, let
\( \mathbf{X} \) and \( \mathbf{Y} \) be separable scaled (ultra)metric spaces, and let
\( \phi_{G} : \mathbf{X} \to G \) and \( \phi_{H} : \mathbf{Y} \to H \) be \emph{surjective\/} Lipschitz morphisms.  By
Proposition \ref{thm:exntesion-of-lipschitz-morphisms} they extend to surjective homomorphisms
\[ \phi_{G} : \clF(\mathbf{X}) \to G, \quad \phi_{H} : \clF(\mathbf{Y}) \to H\]
with kernels \( \hker{G} \) and \( \hker{H} \) respectively.  We note that as proved in \cite[Theorem 3.10]{MR2278689},
for any Polish group \( G \) there are plenty of surjective Lipschitz morphisms \( \phi : \invx \to G \), and moreover,
one may always take \( X = \mathbb{N}^{\mathbb{N}} \).  Note also that \( G \) is isomorphic to
\( \clF(\mathbf{X})/\hker{G} \) with the quotient topology (see, for instance, \cite[Theorem 1.2.6]{MR1425877}).  We
shall identify \( G \) with \( \clF(\mathbf{X})/\hker{G} \) and \( H \) with \( \clF(\mathbf{Y})/\hker{H} \).  By
Corollary \ref{cor:embedding-of-completions} we may view \( \hker{G} \) and \( \hker{H} \) as subgroups of
\( \clF(\mathbf{X} \cup \mathbf{Y}) \).  Let \( \hker{G*H}^{\circ} \) be the normal subgroup of
\( \clF(\mathbf{X} \cup \mathbf{Y}) \) generated by \( \hker{G} \) and \( \hker{H} \):
\[ \hker{G*H}^{\circ} = \Big\{ f_{1}h_{1}f_{1}^{-1} \cdots f_{n}h_{n}f_{n}^{-1} \Bigm | n \in \mathbb{N},\ f_{i} \in
\clF(\mathbf{X} \cup \mathbf{Y}),\ h_{i} \in \big\langle \hker{G}, \hker{H} \big\rangle \Big\}, \]
and let \( \hker{G*H} \) be the closure of \( \hker{G*H}^{\circ} \), i.e., \( \hker{G*H} \) is the \emph{closed\/}
normal subgroup in \( \clF(\mathbf{X} \cup \mathbf{Y}) \) generated by \( \hker{G} \) and \( \hker{H} \).

\begin{lemma}
  \label{lem:pi-of-freepr-kernel-is-phi-kernel}
  In the setting above \( \pi_{X}(\hker{G*H}) = \hker{G} \).
\end{lemma}
\begin{proof}
  If \( g \in \hker{G*H}^{\circ} \) is of the form
  \[ g = f_{1}h_{1}f_{1}^{-1} \cdots f_{n}h_{n}f_{n}^{-1},\quad f_{i} \in \clF(\mathbf{X} \cup
  \mathbf{Y}),\ h_{i} \in \big\langle \hker{G}, \hker{H} \big\rangle, \] then
  \[ \pi_{X}(g) = \pi_{X}(f_{1})\pi_{X}(h_{1})\pi_{X}(f_{1})^{-1} \cdots \pi_{X}(f_{n})\pi_{X}(h_{n})\pi_{X}(f_{n})^{-1}
  \in \hker{G}, \]
  because \( \pi_{X}(h_{i}) \in \hker{G} \), \( \pi_{X}(f_{i}) \in \clF(\mathbf{X}) \) and \( \hker{G} \) is a
  normal subgroup in \( \clF(\mathbf{X}) \).  Thus \( \pi_{X}(\hker{G*H}^{\circ}) = \hker{G} \) and therefore
  also \( \pi_{X}(\hker{G*H}) = \hker{G} \), since \( \hker{G} \) is closed in \( \clF(\mathbf{X}) \), and hence
  also in \( \clF(\mathbf{X} \cup \mathbf{Y}) \).
\end{proof}

A \emph{Polish free product of \( G \) and \( H \) over \( \phi_{G} \) and \( \phi_{H} \)} is the group
\( \clF(\mathbf{X} \cup \mathbf{Y})/\hker{G*H} \); we denote it by \( \frprod{G}{H} \).  The free product
comes with homomorphisms \( \iota_{G} : G \to \frprod{G}{H} \), \( \iota_{H} : H \to \frprod{G}{H} \) and
\( \pi_{G} : \frprod{G}{H} \to G \), \( \pi_{H} : \frprod{G}{H} \to H \) given by
\begin{displaymath}
  \begin{aligned}
    \iota_{G}(f\hker{G}) = f\hker{G*H},&\quad \iota_{H}(f\hker{H}) = f\hker{G*H}, \\
    \pi_{G}(f\hker{G*H}) = \pi_{X}(f)\hker{G},&\quad \pi_{H}(f\hker{G*H}) = \pi_{Y}(f)\hker{H}.
  \end{aligned}
\end{displaymath}
Note that \( \pi_{G} \) and \( \pi_{H} \) are well-defined by Lemma \ref{lem:pi-of-freepr-kernel-is-phi-kernel}.  Note
also that \( \pi_{G}\big(\iota_{G}(g)\big) = g \) and \( \pi_{H}\big(\iota_{H}(h)\big) = h \) for all \( g \in G \) and
\( h \in H \).

\begin{proposition}
  \label{thm:properties-of-the-free-product}
  Let \( \mathbf{X} \), \( \mathbf{Y} \), \( \phi_{G} \), \( \phi_{H} \) and \( \frprod{G}{H} \) be as above.  Let
  \( d_{G} \) and \( d_{H} \) be compatible left invariant (ultra)metrics on \( G \) and \( H \) with respect to which
  \( \phi_{G} \) and \( \phi_{H} \) are Lipschitz morphisms.
  \begin{enumerate}[(i)]
  \item\label{item:injectivity} \( \iota_{G} \) and \( \iota_{H} \) are injective;
  \item\label{item:iota-continuous}  \( \iota_{G} \) and \( \iota_{H} \) are continuous;
  \item\label{item:pi-continuous} \( \pi_{G} \) and \( \pi_{H} \) are continuous;
  \item\label{item:trivial-intersection} \( \iota_{G}(G) \cap \iota_{H}(H) = \{e\} \);
  \item\label{item:density-of-the-image}  \( \langle \iota_{G}(G), \iota_{H}(H) \rangle \) is a dense
    subgroup of \( \frprod{G}{H} \).
  \end{enumerate}
  Recall that \( \cansc{G} \) denotes the canonical scale on \( G \).
  \begin{enumerate}[(i)]
    \setcounter{enumi}{5}
  \item\label{item:universality-property} If \( (T,d_{T}) \) is a Polish (ultra)metric group, \( \psi_{G} : G \to T \),
    \( \psi_{H} : H \to T \) are Lipschitz homomorphisms and
    \( \cansc{T}(\psi_{G}(g),r) \le \cansc{G}(g,r) \), \( \cansc{T}(\psi_{H}(h),r) \le \cansc{H}(h,r) \) for all
    \( g \in G \), \( h \in H \) and \( r \in \mathbb{R}^{+} \), then there exists a unique continuous homomorphism
    \( \psi : \frprod{G}{H} \to T \) such that \( \psi \circ \iota_{G} = \psi_{G} \) and
    \( \psi \circ \iota_{H} = \psi_{H} \).
  \end{enumerate}
\end{proposition}
\begin{proof}
  \eqref{item:injectivity} To show that \( \iota_{G} \) is injective it is enough to check that
  \( \hker{G*H} \cap \clF(\mathbf{X}) = \hker{G} \).  If \( f \in \hker{G*H} \cap \clF(\mathbf{X}) \),
  then \( f = \pi_{X}(f) \in \pi_{X}(\hker{G*H}) = \hker{G} \) by Lemma \ref{lem:pi-of-freepr-kernel-is-phi-kernel};
  hence \( f \in \hker{G} \).

  \eqref{item:iota-continuous} Let \( d \) be a compatible right-invariant (ultra)metric on
  \( \clF(\mathbf{X} \cup \mathbf{Y}) \).  It induces compatible right-invariant (ultra)metrics on the
  factor groups \(  \clF(\mathbf{X} \cup \mathbf{Y}) / \hker{G*H} \) and
  \( \clF(\mathbf{X}) / \hker{G} \) (see \cite[Lemma 2.2.8]{MR2455198})
  \begin{displaymath}
    \begin{aligned}
      d_{1}(f_{1} \hker{G*H}, f_{2} \hker{G*H})  &= \inf\big\{ d(f_{1} k_{1}, f_{2} k_{2}) : k_{1},k_{2} \in
      \hker{G*H} \big\}, \\
      d_{2}(f_{1} \hker{G}, f_{2} \hker{G})  &= \inf\big\{ d(f_{1} k_{1}, f_{2} k_{2}) : k_{1},k_{2} \in \hker{G} \big\}. \\
    \end{aligned}
  \end{displaymath}
  With respect to the (ultra)metrics \( d_{1} \) and \( d_{2} \) the homomorphism \( \iota_{G} \) is Lipschitz, hence
  continuous.

  \eqref{item:pi-continuous} It is enough to prove that \( \pi_{G} \) is continuous at the identity, i.e., that
  \( f_{n}\hker{G*H} \to \hker{G*H} \) implies \( \pi_{X}(f_{n})\hker{G} \to \hker{G} \).  The sequence
  \( f_{n} \hker{G*H} \) converges to \( \hker{G*H} \) if and only if there is a sequence \( h_{n} \in \hker{G*H} \)
  such that \( f_{n}h_{n} \to e \).  This implies \( \pi_{X}(f_{n}) \pi_{X}(h_{n}) \to e \) with
  \( \pi_{X}(h_{n}) \in \hker{G} \) by Lemma \ref{lem:pi-of-freepr-kernel-is-phi-kernel}, and therefore
  \( \pi_{G}(f_{n}\hker{G*H}) = \pi_{X}(f_{n})\hker{G} \to \hker{G}\).

  \eqref{item:trivial-intersection} If \( f\hker{G*H} \in \iota_{G}(G) \cap \iota_{H}(H) \), then
  \( f \hker{G*H} = f_{1} \hker{G*H} = f_{2} \hker{G*H} \) for some \( f_{1} \in \clF( \mathbf{X}) \) and
  \( f_{2} \in \clF(\mathbf{Y}) \).  Therefore
  \( \pi_{X}(f\hker{G*H}) = \pi_{X}(f_{1}\hker{G*H}) = \pi_{X}(f_{2}\hker{G*H}) \), but
  \( \pi_{X}(f_{1}\hker{G*H}) = f_{1}\hker{G} \) and \( \pi_{X}(f_{2}\hker{G*H}) = \hker{G} \), whereby
  \( f_{1} \in \hker{G}\) and thus \( f\hker{G*H} = \hker{G*H} \).

  \eqref{item:density-of-the-image} This item is obvious, since the group generated by the images of \( \iota_{G} \) and
  \( \iota_{H} \) is nothing else but
  \( \big\langle \clF(\mathbf{X}), \clF(\mathbf{Y}) \big\rangle \hker{G*H} \).

  \eqref{item:universality-property} Maps \( \psi_{G} \circ \phi_{G} : \invx \to T \) and \( \psi_{H} \circ \phi_{H} :
  \invy \to T \) are Lipschitz morphisms and so is the map \( \zeta : \invxcupy \to T \) given by
  \begin{displaymath}
    \zeta(z) = 
    \begin{cases}
      \psi_{G} \circ \phi_{G}(z) & \textrm{if \( z \in \invx \)},\\
      \psi_{H} \circ \phi_{H}(z) & \textrm{if \( z \in \invy \)}.\\
    \end{cases}
  \end{displaymath}
  By Proposition \ref{thm:exntesion-of-lipschitz-morphisms} the map \( \zeta \) extends to a continuous homomorphism
  \( \zeta : \clF(\mathbf{X} \cup \mathbf{Y}) \to T \).  Since \( \zeta \) extends both
  \( \psi_{G} \circ \phi_{G} \) and \( \psi_{H} \circ \phi_{H} \), the kernel of \( \zeta \) contains \( \hker{G} \) and
  \( \hker{H} \), and therefore also \( \hker{G*H} \).  Thus \( \zeta \) factors to a
  continuous homomorphism \( \psi : \frprod{G}{H} \to T \).  Uniqueness follows from item
  \eqref{item:density-of-the-image}.
\end{proof}

Note that items \eqref{item:injectivity}, \eqref{item:iota-continuous} and \eqref{item:pi-continuous} imply that
\( \iota_{G} \) and \( \iota_{H} \) are embeddings.  The homomorphisms \( \iota_{G} \), \( \iota_{H} \) can be extended
to a homomorphism from the free product of abstract groups \( \iota : G*H \to \frprod{G}{H} \).  Is the homomorphism
\( \iota \) injective?  We shall show in Corollary \ref{cor:no-cancellations-for-large-enough-scales.} that the answer
is yes when \( \phi_{G} \) and \( \phi_{H} \) are ``large enough''.

\begin{lemma}
  \label{lem:emdedding-of-Polish-groups-non-trivial-on-f}
  Let \( f \in G*H \) be a non-trivial element in the abstract free product.  There are a Polish group \( T \) and two
  embeddings \( \psi_{G} : G \to T \) and \( \psi_{H} : H \to T \) such that for the common extension of these
  homomorphisms \( \psi : G * H \to T \) one has \( \psi(f) \ne e \).  Moreover, if \( G \) and \( H \) admit compatible
  left invariant ultrametrics, then we may find \( T \) that also admits a compatible left invariant ultrametric.
\end{lemma}

\begin{proof}
  Here we deal with metric and ultrametric cases separately.  First assume that \( G \) and
  \( H \) are general Polish groups.

  Let \( f \in G*H \) be given.  By conjugating \( f \) if necessary we may assume without loss of generality that
  \( f \) ``starts with \( g \)'', i.e, it is of the form \( f = g_{n-1}h_{n-1} \cdots g_{0}h_{0} \) for some
  \( n \ge 1 \) with non-trivial \( g_{i} \in G \) and \( h_{j} \in H \).  By a theorem of Uspenskij \cite{MR847156}
  (see also \cite[Theorem 9.18]{MR1321597}) the group \( \homgilcube \) of homeomorphisms of the Hilbert Cube with the
  compact-open topology is a universal Polish group in the following sense: any Polish group can be embedded into
  \( \homgilcube \).  In particular, \( G \) and \( H \) can be embedded into \( \homgilcube \); to simplify notations
  we assume that \( G \) and \( H \) are actual subgroups of \( \homgilcube \).  Note that \( \alpha H \alpha^{-1} \) is
  a copy of \( H \) inside \( \homgilcube \) for any \( \alpha \in \homgilcube \).  To prove the lemma it is therefore
  sufficient to construct a homeomorphism \( \alpha \in \homgilcube \) such that for some \( x_{0} \in \gilcube \)
  \[ g_{n-1}\alpha h_{n-1} \alpha^{-1} \cdots g_{0} \alpha h_{0} \alpha^{-1} (x_{0}) \ne x_{0}.  \]

  Pick any \( x_{0} \in \gilcube \) and any \( x_{1} \in \gilcube \) such that \( x_{1} \ne x_{0} \) and
  \( h_{0}(x_{1}) \not \in \{x_{1}, x_{0}\} \); set \( x_{2} = h_{0}(x_{1}) \).  Pick any \( x_{3} \in \homgilcube \)
  such that \( x_{3} \not \in \{x_{0}, x_{1}, x_{2} \} \) and \( g_{0}(x_{3}) \not \in \{x_{0}, x_{1}, x_{2}, x_{3}\}
  \); set \( x_{4} = g_{0}(x_{3}) \).  We continue in this fashion and construct a sequence
  \( (x_{k})_{k=1}^{4n} \) such that
  \begin{enumerate}[(i)]
  \item \( x_{i} \ne x_{j} \) for \( i \ne j \);
  \item \( h_{k}(x_{4k+1}) = x_{4k+2} \) for \( k=0, \ldots, n-1 \);
  \item \( g_{k}(x_{4k+3}) = x_{4k+4} \) for \( k=0, \ldots, n-1 \).
  \end{enumerate}

  \begin{figure}[h]
    \includegraphics{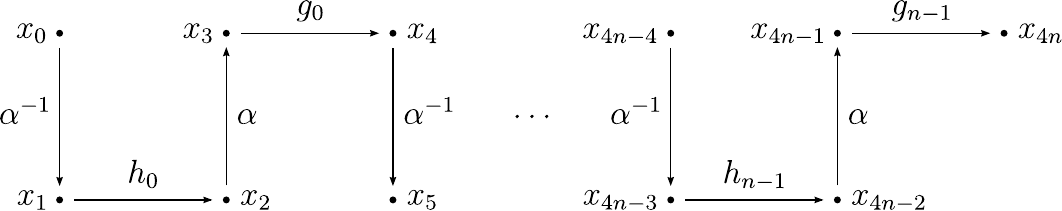}
    \caption{Construction of the homeomorphism \( \alpha \).}
    \label{fig:construction-of-alpha}
  \end{figure}

  For all \( m \in \mathbb{N} \) the space \( \gilcube \) is \( m \)-homogeneous: for tuples
  \( (y_{1}, \ldots, y_{m}) \) and \( (z_{1}, \ldots, z_{m}) \) of distinct elements there is a homeomorphism
  \( \alpha \in \homgilcube \) such that \( \alpha(y_{i}) = z_{i} \) (see, for example, \cite[Exercise 2,
  p. 261]{MR977744}).  Whence there is some \( \alpha \in \homgilcube \) such that \( \alpha(x_{4k+1}) = x_{4k} \) and
  \( \alpha(x_{4k+2}) = x_{4k+3} \) for all \( k = 0, \ldots, n-1 \).  For such an \( \alpha \) we have
  \[ g_{n-1}\alpha h_{n-1} \alpha^{-1} \cdots g_{0} \alpha h_{0} \alpha^{-1} (x_{0}) = x_{4n}, \]
  and \( x_{4n} \ne x_{0} \) by construction.

  In the ultrametric setting a similar argument works with the following modifications.  For the injectively universal
  group we take \( S_{\infty} \) (see \cite[Theorem 1.5.1]{MR1425877})and note that without loss of generality we may
  assume that \( G, H < S_{\infty} \), and every non-trivial element in \( G \) and \( H \) has infinite support.  This
  is so, because we can embed diagonally \( S_{\infty} \) into \( \prod_{n\in \mathbb{N}} S_{\infty} \) and view the
  latter again as a subgroup of \( S_{\infty} \) by partitioning the natural numbers into infinitely many infinite pieces.

  Now for \( f = g_{n-1}h_{n-1} \cdots g_{0}h_{0} \) in \( G*H \) using that supports of \( g_{i} \) and \( h_{j} \) are
  all infinite, we can easily find \( \alpha \in S_{\infty} \) such that
  \[ g_{n-1}\alpha h_{n-1} \alpha^{-1} \cdots g_{0} \alpha h_{0} \alpha^{-1} \ne e.  \]
  Such an element \( \alpha \) is again constructed as shown in Figure \ref{fig:construction-of-alpha}.
\end{proof}

\begin{theorem}
  \label{thm:Polish-groups-generate-free-product}
  There are a Polish group \( T \) and embeddings
  \( \psi_{G} : G \hookrightarrow T \), \( \psi_{H} : H \hookrightarrow T \) such that the group
  \( \langle \psi_{G}(G), \psi_{H}(H) \rangle \) is naturally isomorphic to the group \( G * H \).  Moreover, if \( G \)
  and \( H \) admit compatible left invariant ultrametrics, then \( T \) can be chosen to also admit a compatible left
  invariant ultrametric.
\end{theorem}
\begin{proof}
  Lemma \ref{lem:emdedding-of-Polish-groups-non-trivial-on-f} implies that for any non-trivial \( f \in G*H \) we may
  fix a Polish group \( T_{f} \) and a homomorphism \( \psi_{f} : G*H \to T_{f}\) such that
  \( \psi_{f}|_{G} : G \to T_{f}\) and \( \psi_{f}|_{H} : H \to T_{f} \) are embeddings and \( \psi_{f}(f) \ne e \).

  Let \( f \in G*H \) be given and assume that \( f \) has form \( g_{n-1}h_{n-1} \cdots g_{0} h_{0} \) for some
  non-trivial \( g_{i} \in G \), \( h_{j} \in H \) and \( n \ge 1 \).  By continuity of \( \psi_{f}|_{G} \) and
  \( \psi_{f}|_{H} \), and since \( \psi_{f}(f) \ne e \), there are neighbourhoods \( U^{(f)}_{i} \subseteq G \) of
  \( g_{i} \) and \( V^{(f)}_{j} \subseteq H \) of \( h_{j} \) such that
  \( e \not \in \psi_{G}(U^{(f)}_{n-1}) \psi_{H}(V^{(f)}_{n-1}) \cdots \psi_{G}(U^{(f)}_{0}) \psi_{H}(V^{(f)}_{0}) \).
  Therefore we can select a countable family \( (f_{m})_{m=1}^{\infty} \) of elements \( f_{m} \in G*H \) such that for
  any \( f \in G*H \) there is some \( m \) with \( \psi_{f_{m}}(f) \ne e \).

  Let \( T = \prod_{m} T_{f_{m}} \) be the direct product of the groups \( T_{f_{m}} \) and let \( \psi : G*H \) be 
  given by \( \psi(f)(m) = \psi_{f_{m}}(f) \).  The homomorphisms \( \psi|_{G} : G \to T \) and
  \( \psi|_{H} : H \to T \) are embeddings.  By the choice of the family \( (f_{m}) \) we also have \( \psi(f) \ne e \)
  for any non-trivial \( f \in G*H \) and therefore \( \psi \) is injective.
\end{proof}

\begin{corollary}
  \label{cor:no-cancellations-for-large-enough-scales.}
  There are left invariant compatible (ultra)metrics \( d_{G} \) and \( d_{H} \) on \( G \) and \( H \) respectively and
  scales \( \Gamma_{G} \) and \( \Gamma_{H} \) on \( (G,d_{G}) \) and \( (H, d_{H}) \) with the following property: if
  \( \mathbf{X} \) and \( \mathbf{Y} \) are (ultra)metric scaled spaces and \( \phi_{G} : \mathbf{X} \to G \),
  \( \phi_{H} : \mathbf{Y} \to H \) are surjective Lipschitz morphisms with respect to the scales \( \Gamma_{G} \) and
  \( \Gamma_{H} \), then the canonical homomorphism \( \iota : G*H \to \frprod{G}{H} \) is injective.
\end{corollary}

\begin{proof}
  By Theorem \ref{thm:Polish-groups-generate-free-product} we may assume that \( G \) and \( H \) are closed subgroups
  of a Polish group \( T \) and that \( \langle G, H \rangle \) is isomorphic to \( G*H \).  Let \( d \) be
  a compatible left invariant (ultra)metric on \( T \) and let \( d_{G} \) and \( d_{H} \) be the restrictions of
  \( d \) onto \( G \) and \( H \) respectively.  Finally, let \( \Gamma_{G} \) and \( \Gamma_{H} \) be the restrictions
  of \( \cansc{T} \) onto \( G \) and \( H \).  By item \eqref{item:universality-property} of Proposition
  \ref{thm:properties-of-the-free-product} the maps \( \phi_{G} \) and \( \phi_{H} \) extend to a homomorphism
  \( \phi : \frprod{G}{H} \to T \).  Since \( \langle G, H \rangle \cong G*H\), the homomorphism
  \( \iota : G*H \to \frprod{G}{H}\) must be injective.
\end{proof}

\bibliographystyle{alpha}
\bibliography{references}

\providecommand{\noopsort}[1]{}\def\cprime{$'$} \def\cprime{$'$}
\begin{thebibliography}{DG07b}

\bibitem[BK96]{MR1425877}
Howard Becker and Alexander~S. Kechris.
\newblock {\em The descriptive set theory of {P}olish group actions}, volume
  232 of {\em London Mathematical Society Lecture Note Series}.
\newblock Cambridge University Press, Cambridge, 1996.

\bibitem[DG07a]{MR2332614}
Longyun Ding and Su~Gao.
\newblock Graev metric groups and {P}olishable subgroups.
\newblock {\em Adv. Math.}, 213(2):887--901, 2007.

\bibitem[DG07b]{MR2278689}
Longyun Ding and Su~Gao.
\newblock New metrics on free groups.
\newblock {\em Topology Appl.}, 154(2):410--420, 2007.

\bibitem[Din12]{MR2970459}
Longyun Ding.
\newblock On surjectively universal {P}olish groups.
\newblock {\em Adv. Math.}, 231(5):2557--2572, 2012.

\bibitem[Gao09]{MR2455198}
Su~Gao.
\newblock {\em Invariant descriptive set theory}, volume 293 of {\em Pure and
  Applied Mathematics (Boca Raton)}.
\newblock CRC Press, Boca Raton, FL, 2009.

\bibitem[Gao13]{MR3028617}
Su~Gao.
\newblock Graev ultrametrics and surjectively universal non-{A}rchimedean
  {P}olish groups.
\newblock {\em Topology Appl.}, 160(6):862--870, 2013.

\bibitem[Gra51]{MR0038357}
Mark~I. Graev.
\newblock Free topological groups.
\newblock {\em Amer. Math. Soc. Translation}, 1951(35):61, 1951.

\bibitem[Kec94]{MR1288299}
Alexander~S. Kechris.
\newblock Topology and descriptive set theory.
\newblock {\em Topology Appl.}, 58(3):195--222, 1994.

\bibitem[Kec95]{MR1321597}
Alexander~S. Kechris.
\newblock {\em Classical descriptive set theory}, volume 156 of {\em Graduate
  Texts in Mathematics}.
\newblock Springer-Verlag, New York, 1995.

\bibitem[Slu12]{1111.1538}
Konstantin Slutsky.
\newblock Graev metrics on free products and {HNN} extensions.
\newblock {\em Trans. Amer. Math. Soc., to appear}, 2012.

\bibitem[SU87]{MR913066}
Olga~V. Sipacheva and Vladimir~V. Uspenski{\u\i}.
\newblock Free topological groups with no small subgroups, and {G}raev metrics.
\newblock {\em Vestnik Moskov. Univ. Ser. I Mat. Mekh.}, (4):21--24, 101, 1987.

\bibitem[Usp86]{MR847156}
Vladimir~V. Uspenski{\u\i}.
\newblock A universal topological group with a countable basis.
\newblock {\em Funktsional. Anal. i Prilozhen.}, 20(2):86--87, 1986.

\bibitem[vM89]{MR977744}
Jan van Mill.
\newblock {\em Infinite-dimensional topology}, volume~43 of {\em North-Holland
  Mathematical Library}.
\newblock North-Holland Publishing Co., Amsterdam, 1989.
\newblock Prerequisites and introduction.

\end{thebibliography}

\end{document}